\documentclass[a4paper,12pt]{article}
\usepackage[top=2.5cm,bottom=2.5cm,left=2.5cm,right=2.5cm]{geometry}
\usepackage{cite, amsmath, amssymb}
\usepackage{amssymb}
\usepackage{graphicx}
\newenvironment{proof}{{\noindent \it Proof.}}{\hfill $\blacksquare$\par}
\usepackage{latexsym}
\usepackage{caption}
\usepackage{textcomp}
\usepackage{graphicx,booktabs,multirow}
\usepackage{tikz}
\usepackage{bigdelim}
\usepackage{nicematrix}
\usepackage{mathtools}
\usepackage{comment}
\usepackage{makecell}
\usetikzlibrary{decorations.pathreplacing}
\usetikzlibrary{intersections}
\usepackage{enumerate}
\usepackage{graphicx,booktabs,multirow}
\usepackage{appendix}
\newtheorem{theorem}{Theorem}[section]
\newtheorem{proposition}[theorem]{\rm\bfseries Proposition}
\newtheorem{Observation}[theorem]{\rm\bfseries Observation}
\newtheorem{lemma}[theorem]{Lemma}
\newtheorem{corollary}[theorem]{\rm\bfseries Corollary}
\newtheorem{remark}[theorem]{Remark}
\newtheorem{conjecture}[theorem]{Conjecture}
\newtheorem{problem}[theorem]{Problem}
\newtheorem{Proof of Theorem 1.7.}[theorem]{Proof of Theorem 1.7.}

\newtheorem{definition}[theorem]{Definition}
\usepackage[section]{placeins}
\def\NAT@def@citea{\def\@citea{\NAT@separator}}% Suppress spaces between citations using natbib.sty
\allowdisplaybreaks

\begin{document}
	\vspace*{10mm}
	
	\noindent
	{\Large \bf  The connectedness of friends-and-strangers graphs about graph parameters and others}
	
	\vspace*{7mm}
	
	\noindent
	{\large \bf Xinghui Zhao$^1$,  Lihua You*, Jifu Lin$^2$, Xiaoxue Zhang$^3$}
	\noindent
	
	\vspace{7mm}
	
	\noindent
	School of Mathematical Sciences, South China Normal University,  Guangzhou, 510631, P. R. China.

	\noindent
	e-mail: {\tt 2022021990@m.scnu.edu.cn (X. Zhao)},\quad{\tt ylhua@scnu.edu.cn (L. You)},
	
	 \quad\quad{\tt 2023021893@m.scnu.edu.cn (J. Lin)}, \quad{\tt zhang\_xx1209@163.com (X. Zhang)}.
	 
	 \noindent
	 $^*$ Corresponding author
	\noindent
	
	%\footnotesize $^1${\it School of Mathematical Sciences, South China Normal University, Guangzhou, 510631, P. R. China}\\
	%\footnotesize $^2${\it Department of Mathematics Teaching, Guangzhou Civil Aviation College, Guangzhou, 510403, P. R. China}\\
	%\noindent
	% $^2${\it Department of Mathematics Teaching, Guangzhou Civil Aviation College, Guangzhou, 510403, P. R. China\/} \\
	\vspace{7mm}

	\noindent
	{\bf Abstract} \ 
	\noindent
    Let $X$ and $Y$ be two graphs of order $n$. The friends-and-strangers graph $\textup{FS}(X,Y)$ of $X$ and $Y$ is a graph whose vertex set consists of all bijections $\sigma: V(X)\rightarrow V(Y)$, in which two  bijections $\sigma$ and $ \sigma'$ are adjacent if and only if  they agree on all but two adjacent vertices of $X$ such that the corresponding images are adjacent in $Y$. The most fundamental question about these friends-and-strangers graphs is whether they are connected. In this paper, we provide a sufficient condition regarding the maximum degree $\Delta(X)$ and vertex connectivity $\kappa(Y)$ that ensures the graph $\textup{FS}(X,Y)$ is $s$-connected. As a corollary, we improve upon a result by Bangachev and partially confirm  a conjecture he proposed. Furthermore, we completely characterize the connectedness of $\textup{FS}(X,Y)$, where $X\in\textup{DL}_{n-k,k}$.
	\\[2mm]
	%\vspace{5mm}
	

	\baselineskip=0.30in
	
	\section{Introduction}
	
	 \hspace{1.5em}Throughout this paper, we consider  simple  graphs. Let $G=(V(G),E(G))$ be a graph with vertex set $V(G)$ and edge set $E(G)$.  We write  $uv \in E(G)$ if the vertices $u$ and $v$ are adjacent in $G$, and $uv \not\in E(G)$ otherwise. Let $N_G(u)$  be the neighbourhood set of $u$ in $G$, $N_G[u]=N_G(u)\cup \{u\}$. Then the degree of the vertex $v$ is equal to $|N(v)|$ (the cardinality of the set $N(v)$), denoted by $d_G(v)$ ($d(v)$ for short). We use $\delta(G)$ and $\Delta(G)$  to denote the minimum degree and the maximum degree. Let $P_n$, $C_n$, $K_n$, $S_n$ and $S_n^+$ denote a path,  a cycle, a complete graph,  a star  of order $n$ and a star with $n$ vertices adding an extra edge, respectively. In this paper, we use $K_n-t e$ to denote the graph obtained from $K_n$ by deleting $t$ vertex-disjoint edges, where $1\leq t\leq \lfloor \frac{n}{2}\rfloor$.  The lollipop graph $\textup{Lollipop}_{n-k,k}$  of order $n$ is obtained by identifying one end of  $P_{n-k+1}$ with a vertex of $K_k$. Then dandelion graph $\textup{Dand}_{n-k,k}$ of order $n$ is obtained by identifying one end of  $P_{n-k+1}$ with center vertex of $S_k$. We define the graph set $\textup{DL}_{n-k,k}$, where  each graph  $G\in \textup{DL}_{n-k,k}$ is a spanning subgraph of $\textup{Lollipop}_{n-k,k}$, and $\textup{Dand}_{n-k,k}$ is a spanning subgraph of $G$. A graph $G$ is $s$-connected if the resulting graph is still connected by removing any $s-1$ vertices from $G$. The vertex connectivity (simply  connectivity)  of $G$  is the maximum value of $s$ for which $G$ is  $s$-connected, denoted by $\kappa(G)$.
	
	The friends-and-strangers graphs were introduced by Defant and Kravitz \cite{dc}, which are defined as follows.
	
	\begin{definition}{\rm(\!\!\cite{dc})}\label{d1}
		Let $X$ and $Y$ be two graphs, each with $n$ vertices. The friends-and-strangers graph $\textup{FS}(X,Y)$ of $X$ and $Y$ is a graph with vertex set consisting of all bijections $\sigma: V(X)\rightarrow V(Y)$, two such bijections $\sigma$, $\sigma'$ are adjacent if and only if there exist two distinct vertices $a,b\in V(X)$ such that:
		
		$\bullet$ $ab\in E(X)$ and $\sigma(a)\sigma(b)\in E(Y)$;
		
		$\bullet$  $\sigma(a)=\sigma'(b)$,  $\sigma(b)=\sigma'(a)$ and  $\sigma(c)=\sigma'(c)$ for all $c\in V(X)\setminus \{a,b\}$.
	\end{definition}
	
   	The friends-and-strangers graph $\textup{FS}(X,Y)$ has the following non-technical interpretation. View $V(X)$ as $n$ different positions and $V(Y)$ as $n$ different people. Two people are friends if and only if they are adjacent in $Y$, and  two positions are  adjacent if and only if they are adjacent in $X$. A bijection from $V(X)$ to $V(Y)$ represents $n$ people standing on these $n$ positions such that each person stands on  exactly one position. At each point in time, two people can swap their positions if and only if they are friends and the two positions they stand in are adjacent. When multiple swaps are permitted, a natural question is how one can transition between different configurations. In particular, under what conditions can any two given configurations be reached from one to another through a finite number of swaps.
	
	A classic example of friends-and-strangers graphs is the famous 15-puzzle: 15 tiles labeled with numbers 1 to 15 and an empty tile are placed on a $4\times4$ grid. At each move, two tiles with numbers are not allowed to exchange positions, but the empty tile can exchange positions with any adjacent tile. By the definition of friends-and-strangers graphs, the game corresponds to	 $\textup{FS}(S_{16},G_{4\times4})$, where $G_{4\times4}$ is a $4\times4$ grid graph and the center vertex of the $S_{16}$ corresponds to the empty tile. Generalizing  the 15-puzzle, Wilson \cite{w} has provided the sufficient and necessary conditions for $\textup{FS}(S_n,Y)$ to be connected.
	
	\begin{theorem}{\rm(\!\!\cite{w})}\label{th1}
		Let $Y$ be a graph of order $n\geq4$. Then $\textup{FS}(S_n,Y)$ is connected if and only if $Y$ is $2$-connected, non-bipartite and  $Y\not\in\{ C_n,\theta_0\}$, where $\theta_0$ is shown in Figure \ref{Figure 1}. Furthermore, if $Y$ is a bipartite,  $2$-connected  graph, and $Y\not\cong C_n$, then $\textup{FS}(S_n,Y)$ has exactly two connected components.
	\end{theorem}
	
\begin{figure}[!h]
	\centering
	\begin{tikzpicture}
		\node[anchor=south west,inner sep=0] (image) at (0,0) {\includegraphics[width=0.6\textwidth]{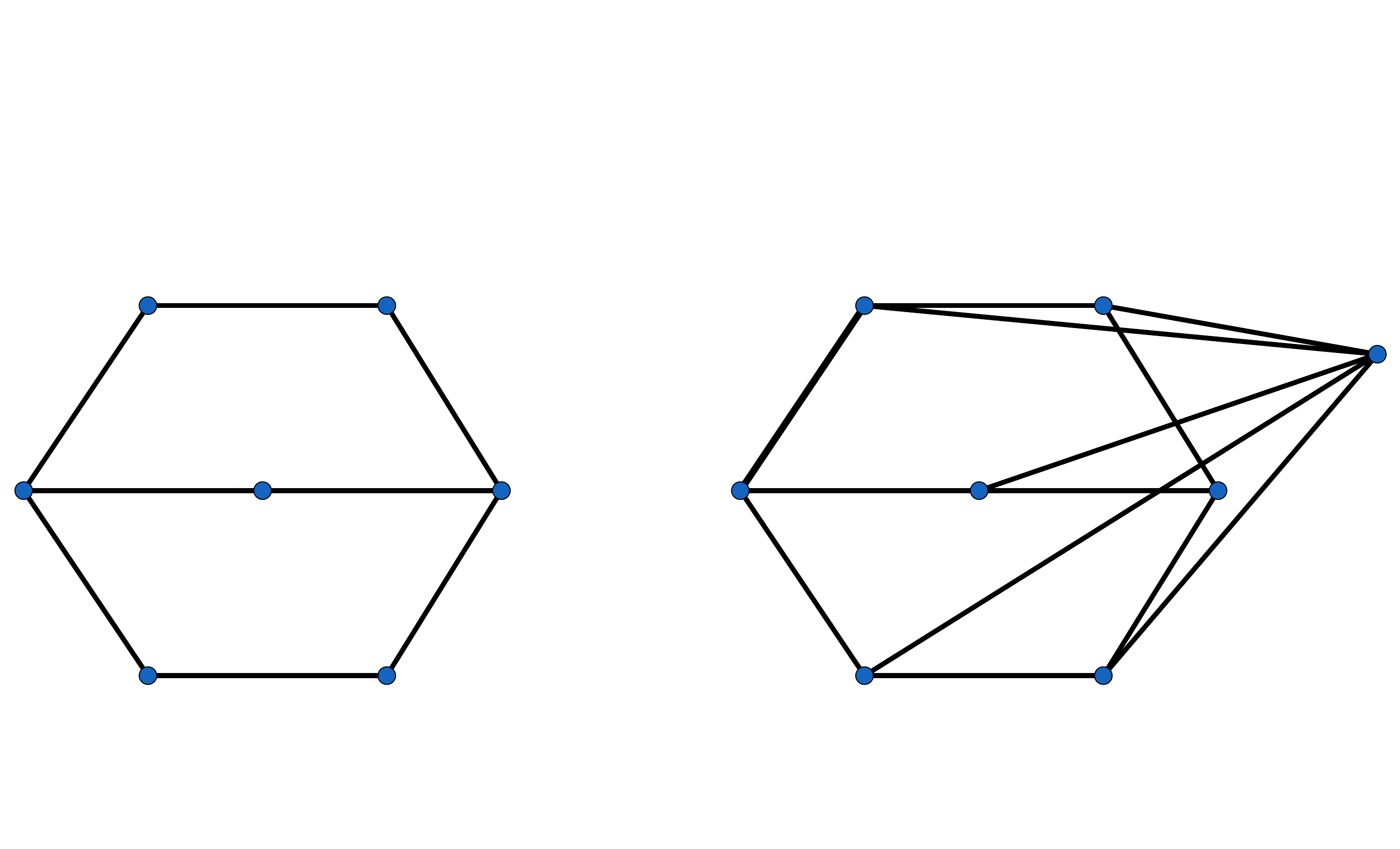}};
		\begin{scope}[
			x={(image.south east)},
			y={(image.north west)}
			]
			\node [black, font=\bfseries] at (0.2,0) {$\theta_0$ };
			\node [black, font=\bfseries] at (0.7,0) {$\theta_1$ };

		\end{scope}
	\end{tikzpicture}
	
	\caption{Graphs $\theta_0$ and $\theta_1$.}\label{Figure 1}
\end{figure}
	
	The questions and results in literature on the friends-and-strangers graph $\textup{FS}(X,Y)$ can be roughly divided into three categories: Both $X$ and $Y$ are random graphs \cite{Al,a,wang}, or at least one of $X$, $Y$ is a concrete graph, such as paths, cycles, lollipop graphs, spiders \cite{wang2,wang3,zhu,dc,dc2}, or $X$ and $Y$ have some extremal structures, such as a lower bound on their respective minimum degree \cite{Al,B,k}.
	
	In this paper, we address questions in the second and third categories.
	
	In 2022, Bangachev \cite{B} studied the minimum degree condition for $X$ and $Y$ that guarantees the connectedness of $\textup{FS}(X,Y)$, and  proposed   Conjecture \ref{c1}.
	
	\begin{theorem}{\rm(\!\!\cite{B})}\label{th2}
		Suppose that $X$ and $Y$ are two connected graphs on $n$ vertices satisfying $\min(\delta(X),\delta(Y))+2\max(\delta(X),\delta(Y))\geq2n$. Then $\textup{FS}(X,Y)$ is connected.
	\end{theorem}
	
	\begin{theorem}{\rm(\!\!\cite{B})}\label{th3}
		Suppose that $X$ and $Y$ are two graphs on $n\geq 6$ vertices satisfying $\delta(X),\delta(Y)>\frac{n}{2}$ and $2\min(\delta(X),\delta(Y))+3\max(\delta(X),\delta(Y))\geq3n$. Then $\textup{FS}(X,Y)$ is connected.
	\end{theorem}

	\begin{conjecture}{\rm(\!\!\cite{B})}\label{c1}
		Suppose that $X$ and $Y$ are two connected n-vertex graphs satisfying the condition $2\min(\delta(X),\delta(Y))+3\max(\delta(X),\delta(Y))\geq3n$. Then $\textup{FS}(X,Y)$ is connected.
	\end{conjecture}
	
	In fact, by $\textup{FS}(X,Y)\cong\textup{FS}(Y,X)$ \rm(\!\!\cite{dc}), we can always let $\delta(Y)\geq \delta(X)$ in Theorems \ref{th2}, \ref{th3} and Conjecture \ref{c1}. In this paper, we studied some graph parameters condition for $X$ and $Y$ that guarantees the connectedness of $\textup{FS}(X,Y)$.
	
		\begin{theorem}\label{mt21}
		Let $s\geq2$ be an integer, $X$ and $Y $ be two connected bipartite graphs of order $n$. Then
		
		{\rm(i)} if $\Delta(X)+\kappa(Y)\geq n+s-1$ and $Y\not\cong C_n$, then $\textup{FS}(X,Y)$  has exactly two connected components and each connected component is  $s$-connected;
		
		{\rm(ii)}  for any integers $\Delta$ and $\kappa$ that satisfy $\Delta+\kappa\leq n$, $\Delta\leq n-1$, and $1\leq\kappa\leq \lfloor \frac{n}{2}\rfloor$, there exists a connected bipartite graph $X$ such that $\Delta(X)=\Delta$,  for any connected bipartite graph $Y$ with $\kappa(Y)=\kappa$, the graph $\textup{FS}(X,Y)$ has at least three connected components.
	\end{theorem}
	
%	\begin{theorem}\label{mt1}
%		Let $X$ and $Y$ be two connected graphs of order $n$, $\Delta(X)+\kappa(Y)\geq n+1$. Then 
%		
%		{\rm(i)} if $Y$ is non-bipartite and  $Y\not\cong C_n,\theta_0$,  then  $\textup{FS}(X,Y)$ is connected.
%		
%		 {\rm(ii)} if  $\kappa(Y)\geq \lceil \frac{n+1}{2}\rceil$ and $n\geq5$,  then $\textup{FS}(X,Y)$ is connected.
%		 
%		 {\rm(iii)} if there exists $x\in V(X)$ such that $d_X(x)=\Delta(X)$, $N_X(x)$ is not a independent set, $Y\not\cong C_n,\theta_0$,  then $\textup{FS}(X,Y)$ is connected.
%	\end{theorem}

\begin{theorem}\label{mt1}
	Let $s\geq2$ be an integer, $X$, $Y$ be two connected graphs of order $n$, and among which at least one is non-bipartite, $Y\not\in\{ C_n,\theta_0\}$, $\Delta(X)+\kappa(Y)\geq n+s-1$. Then  $\textup{FS}(X,Y)$ is $s$-connected.
\end{theorem}

	As an application of Theorem \ref{mt1}, we obtain the following two conclusions.

	\begin{theorem} \label{mt2}
		Let $s\geq2$ be an integer, $X$ and $Y$ be two connected graphs of order $n (\geq6)$. Then  we have
		
	{\rm(i)} if $Y\cong K_n$, then	$\textup{FS}(X,Y)$ is $\Delta(X)$-connected.
	
	{\rm(ii)} if $Y\not\cong K_n$, $\delta(Y)\geq \delta(X)$, $\Delta(X)+2\delta(Y)\geq2n+s-2$, then $\textup{FS}(X,Y)$ is $(s+1)$-connected.
	\end{theorem}
	
		Theorem \ref{mt2} not only improves upon Theorem \ref{th2}, but also proves   Conjecture \ref{c1} is correct when $\Delta(X)\geq \frac{4}{3}\delta(X)$, which plugs $s=2$ as shown below.
		
		\begin{theorem} \label{mt3}
		Let $s\geq2$ be an integer, $X$ and $Y$ be two connected graphs of order $n (n\geq6)$, $\delta(Y)\geq \delta(X)$, $\Delta(X)\geq \frac{4}{3}\delta(X)$, $2\delta(X)+3\delta(Y)\geq3n+\frac{3s}{2}-3$. Then  $\textup{FS}(X,Y)$ is $s$-connected.
	\end{theorem}

%	In 2021, Defant and  Kravitz \cite{dc} has proposed the following conjecture.
	
%	\begin{conjecture}\label{c3}
%		Let $Y$ be a graph of order $n$ such that $\overline{Y}$ is a forest consisting of trees $\mathcal{T}_1,\cdots,\mathcal{T}_r$ such that the greatest common divisor of $|V(\mathcal{T}_1)|,\cdots,|V(\mathcal{T}_{r-1})|$ and $|V(\mathcal{T}_{r})|$ is $1$. If $X$ is a $2$-connected graph on $n$ vertices, then $\textup{FS}(X,Y)$ is connected.
%	\end{conjecture}

		In 2021, Defant and  Kravitz \cite{dc} showed that $\textup{FS}(\textup{Lollipop}_{n-3,3},Y)$ is connected if and only if  $\textup{FS}(\textup{Dand}_{n-3,3},Y)$  is connected. In 2023, Wang and Chen \cite{wang3} obtained the sufficient and necessary conditions for $\textup{FS}(\textup{Lollipop}_{n-k,k},Y)$ to be connected, and   proposed the following problem for further research.
	\begin{comment}
	\begin{theorem}{\rm(\!\!\cite{wang3})}\label{th4}
		Let $2\leq k\leq n$ be integers and $Y$ be a graph on $n$ vertices. Then the graph $\textup{FS}(\textup{Lollipop}_{n-k,k},Y)$ is connected if and only if $Y$ is $(n-k+1)$-connected.
	\end{theorem}
	
	By  Theorem \ref{mt1} and Theorem \ref{th4}, we have the following corollary.
	
	\begin{corollary}\label{c2}
		Let $2\leq k\leq n$ be integers and $Y$ be a graph on $n$ vertices. Then the graph $\textup{FS}(\textup{Lollipop}_{n-k,k},Y)$ is connected if and only if $\textup{FS}(\textup{Lollipop}_{n-k,k},Y)$ is $2$-connected.
	\end{corollary}
	\end{comment}

	\begin{problem}{\rm(\!\!\cite{wang3})}\label{p1}
		For what $k$ and $n$, it holds that $\textup{FS}(\textup{Lollipop}_{n-k,k},Y)$ is connected if and only if $\textup{FS}(\textup{Dand}_{n-k,k},Y)$  is connected?
	\end{problem}
	
	In \cite{wang3}, the authors also studied Problem \ref{p1} and demonstrated that the statement  holds when $n\geq 2k-1$  and does not hold when $n=k$. We completely solve this problem in this paper. Furthermore, as an application of Theorem \ref{mt1}, we  obtain the sufficient and necessary conditions for $\textup{FS}(X,Y)$ to be connected when $X\in\textup{DL}_{n-k,k}$, $k+1\leq n$.

	\section{Preliminaries}\label{sec-pre}

	\hspace*{6mm}In this section, we recall some results which will be used in the next sections.
	
%	\begin{lemma}{\rm(\!\!\cite{dc})}\label{lem2}
%		For any two graphs $X$ and $Y$ of order $n$, the graphs $\textup{FS}(X,Y)$ and $\textup{FS}(Y,X)$ are isomorphic.
%	\end{lemma}
	
	\begin{lemma}{\rm(\!\!\cite{g})}\label{g}
		Let $Y$ be a graph of order $n$. Then $\textup{FS}(K_n,Y)$ is connected if and only if $Y$ is connected.
	\end{lemma}
	
	We use $X\preceq Y$ to denote that $X$ is a spanning subgraph of $Y$.
	
	\begin{lemma}{\rm(\!\!\cite{dc})}\label{lem3}
		Let $X, \widetilde{X}, Y, \widetilde{Y}$ be graphs of order $n$.  Then $\textup{FS}(\widetilde{X},\widetilde{Y})\preceq\textup{FS}(X,Y)$ if $\widetilde{X}\preceq X,\widetilde{Y}\preceq Y$. In particular, $\textup{FS}(X,Y)$ is connected if $\textup{FS}(\widetilde{X},\widetilde{Y})$ is connected. 
	\end{lemma}
	
		\begin{lemma}{\rm(\!\!\cite{dc})}\label{lem31}
	Let $Y$ be a graph of order $n$. Then $\textup{FS}(S_n^+,Y)$ is connected if $Y$ is $2$-connected  and $Y\not\in\{ C_n,\theta_0\}$.
	\end{lemma}
	
	\begin{lemma}{\rm(\!\!\cite{dc})}\label{lem32}
		If both $X$ and $Y$ are bipartite graphs of order $n\geq3$, then $\textup{FS}(X,Y)$ is disconnected.
	\end{lemma}

	\begin{lemma}{\rm(\!\!\cite{dc})}\label{lem34}
		Let $X$ and $Y$ be two bipartite graphs of order $n\geq3$, $V(X)=V(Y)$, $\{A_X,B_X\}$ and $\{A_Y,B_Y\}$ be a bipartition of $X$ and $Y$, respectively. Then $g(\sigma)$ and $g(\sigma')$ have the same parity if $ \sigma$ and $ \sigma'$ are in the same connected component of $\textup{FS}(X,Y)$ for any $ \sigma, \sigma'\in V(\textup{FS}(X,Y))$, where the set $\sigma(A_X)$ represents the image set of all elements in $A_X$ under the permutation $\sigma$, and $$g(\sigma)=|\sigma(A_X)\cap A_Y|+\frac{sgn(\sigma)+1}{2},$$  $$ sgn(\sigma)=\begin{cases}
			1,&\ \text{if}\ \sigma \text{is an even permutation;}\\
			-1,&\ \text{if}\ \sigma \text{is an odd permutation.}
		\end{cases}$$
	\end{lemma}
	
	\begin{lemma}{\rm(\!\!\cite{dc})}\label{lem33}
		Let $Y$ be a graph of order $n\geq3$. Then $\textup{FS}(C_n,Y)$ is connected if and only if $\overline{Y}$ is a forest consisting of trees $\mathcal{T}_1,\cdots,\mathcal{T}_r$ such that the greatest common divisor of $|V(\mathcal{T}_1)|,\cdots,|V(\mathcal{T}_{r-1})|$ and $|V(\mathcal{T}_{r})|$ is $1$.
	\end{lemma}
	
		\begin{lemma}{\rm(\!\!\cite{wang3})}\label{th4}
		Let $2\leq k\leq n$ be an integer and $Y$ be a graph on $n$ vertices. Then the graph $\textup{FS}(\textup{Lollipop}_{n-k,k},Y)$ is connected if and only if $Y$ is $(n-k+1)$-connected.
	\end{lemma}
	
	\begin{lemma}{\rm(\!\!\cite{k})}\label{lem611}
		Let $Y$ be a connected graph of order $n\geq3$. Then the connectivity of connected components of $\textup{FS}(S_n,Y)$ is equal to  the minimum degree of $Y$.
	\end{lemma}

	%	\begin{lemma}{\rm(\!\!\cite{w})}\label{lem34}
%		Let $Y$ be a $2$-connected graph of order $n\geq3$, $Y\not\cong C_n, \theta_0$. If $Y$ is a non-bipartite, then $\textup{FS}(S_n,Y)$ is connected; If $Y$ is a bipartite, then $\textup{FS}(S_n,Y)$ has exactly two connected components.
%	\end{lemma}

	\begin{lemma}{\rm(\!\!\cite{bd})}\label{lem5}
		Let $X (\not\cong K_n)$ be a graph of order $n$. Then $\kappa(X)\geq 2\delta(X)+2-n$.
	\end{lemma}
	
	\section{\textbf{The proof of Theorem \ref{mt21} and Theorem \ref{mt1} }}
	
	\hspace*{6mm} Let $X$ and $Y$ be two graphs of order $n$, $X'$ and $Y'$ be the induced subgraphs  with order $m (< n)$ of $X$ and $Y$, respectively, $V(X')=V(X)\setminus	\{x_1,x_2,\cdots,x_{n-m}\}$,  $V(Y')=V(Y)\setminus	\{y_1,y_2,\cdots,y_{n-m}\}$. We define a bijection $\phi$  from $\{x_1,x_2,\cdots,x_{n-m}\}$ to $\{y_1,y_2,\cdots,y_{n-m}\}$. Clearly, there are $(n-m)!$ different $\phi$'s, denoted by $\phi_1,\phi_2,\cdots,\phi_{(n-m)!}$. For $1\leq i\leq (n-m)!$, we define $V_{\phi_i}$  to be a subset of $V(\textup{FS}(X, Y))$ such that for any $\sigma\in V_{\phi_i}$,  $\sigma(x_j)=\phi_i(x_j)$ for any $1\leq j\leq n-m$. Clearly, we have $V_{\phi_i}=V(\textup{FS}(X, Y)[V_{\phi_i}])$ and  $\bigcup\limits_{i=1}^{(n-m)!}V_{\phi_i}\subset V(\textup{FS}(X, Y))$. Moreover,  for any $\tau\in V(\textup{FS}(X, Y))$, we denote  $\tau|_{X'}$ as the restriction of the mapping $\tau$ to $V(X')$; for any $\varphi\in V(\textup{FS}(X', Y'))$, we denote $\varphi|^{X}$ as the extension of the mapping $\varphi$ to $V(X)$. In fact, there are $(n-m)!$ different extensions of the mapping $\varphi$,  we denote  $\varphi|_i^X$  as the extension corresponding to $\phi_i$, say, $\varphi|_i^X(x_j)=\phi_i(x_j)$ for $1\leq j\leq n-m$, where $1\leq i\leq (n-m)!$. Specially, if $n-m=1$, then  $\varphi|^{X}=\varphi|_1^{X}$ is the unique extension of the mapping $\varphi$.
	
	For any subset $V_1$ of $ V(X)$, denote by $X-V_1$ the induced subgraph of $X$ on the vertex set $V(X)\backslash V_1$ and denote by $X[V_1]$ the induced subgraph of $X$ on the vertex set  $V_1$. Specially, if $V_1=\{x_0\}$, we denote by $X-x_0$ the induced subgraph of $X$ on the vertex set $V(X)\backslash \{x_0\}$.

	\begin{lemma}\label{l63}
		Let $X$ and $Y$ be two graphs of order $n$, $X'$ and $Y'$ be the induced subgraphs  with order $m (\leq n)$ of $X$ and $Y$, respectively. Then $\textup{FS}(X, Y)$ contains $(n-m)!$ vertex-disjoint copies of $\textup{FS}(X', Y')$ as subgraphs.
	\end{lemma}
	\begin{proof}
		Firstly, we have $|V(\textup{FS}(X', Y'))|=|V(\textup{FS}(X, Y)[V_{\phi_i}])|$ for any $i\in\{1,2,\cdots,(n-m)!\}$, where $V_{\phi_i}$ is defined as above. 
		
		Let $f$ be a mapping from $\textup{FS}(X', Y')$ to $ \textup{FS}(X, Y)[V_{\phi_i}]$ for given $i\in\{1,2,\cdots,(n-m)!\}$ such that $f(\varphi)=\sigma$, where $\varphi\in V(\textup{FS}(X', Y'))$ and  $\sigma=\varphi|_i^X\in V_{\phi_i}$. It is easy to check that $f$ is a bijection. Now we show the map $f$ preserves the structure of $\textup{FS}(X', Y')$.
		
		For any $\varphi_1, \varphi_2\in V(\textup{FS}(X', Y'))$, there must exist $\sigma_1, \sigma_2\in V_{\phi_i}$ such that $\sigma_1=\varphi_1|_i^X, \sigma_2=\varphi_2|_i^X$ and $\sigma_1(x_j)=\sigma_2(x_j)$ for any $j\in \{1,\cdots,n-m\}$. If $\varphi_1 \varphi_2\in E(\textup{FS}(X', Y'))$, say, there exist $a,b\in V(X')$ satisfying  $ab\in E(X')$,  $\varphi_1(a)\varphi_1(b)\in E(Y')$, and $\varphi_1(a)=\varphi_2(b)$, $\varphi_1(b)=\varphi_2(a)$, and $\varphi_1(c)=\varphi_2(c)$ for any $c\in V(X')\setminus \{a,b\}$ by Definition \ref{d1}, then $a,b\in V(X')\subset V(X)$ satisfying $ab\in E(X')\subset E(X)$, $\sigma_1(a)\sigma_1(b)\in E(Y)$, and $\sigma_1(a)=\sigma_2(b)$, $\sigma_1(b)=\sigma_2(a)$, and $\sigma_1(d)=\sigma_2(d)$ for any $d\in V(X)\setminus \{a,b\}$ by $\sigma_1, \sigma_2\in V_{\phi_i}$, which implies $f(\varphi_1)f(\varphi_2)=\sigma_1 \sigma_2\in E(\textup{FS}(X, Y)[V_{\phi_i}])$.
		
		On the other hand, for any $\sigma_1 \sigma_2\in E(\textup{FS}(X, Y)[V_{\phi_i}])$, where $\sigma_1,  \sigma_2\in V_{\phi_i}$, we have $\sigma_1(x_j)=\sigma_2(x_j)$ for any $j\in\{1,\cdots,n-m\}$. By Definition \ref{d1}, there exist $a,b\in V(X)$ satisfying  $ab\in E(X)$, $\sigma_1(a)\sigma_1(b)\in E(Y)$, and $\sigma_1(a)=\sigma_2(b)$, $\sigma_1(b)=\sigma_2(a)$, and $\sigma_1(d)=\sigma_2(d)$ for any $d\in V(X)\backslash\{a,b\}$. Clearly,  $a,b\notin \{x_1,\cdots,x_{n-m}\}$ by $\sigma_1(a)\neq \sigma_2(a)$ and  $\sigma_1(b)\neq \sigma_2(b)$, say, $a,b\in V(X')$. Thus we have $ab\in E(X')$, $\sigma_k|_{X'}(a)=\sigma_k(a)\in Y'$, $\sigma_k|_{X'}(b)=\sigma_k(b)\in Y'$ for $k\in\{1,2\}$, $\sigma_k|_{X'}(a)\sigma_k|_{X'}(b)=\sigma_k(a)\sigma_k(b)\in E(Y')$ for $k\in\{1,2\}$,  $\sigma_1|_{X'}(a)=\sigma_2|_{X'}(b)$, $\sigma_1|_{X'}(b)=\sigma_2|_{X'}(a)$, and $\sigma_1|_{X'}(c)=\sigma_2|_{X'}(c)$ for any $c\in V(X')\backslash\{a,b\}$. Therefore, $f^{-1}(\sigma_1)f^{-1}(\sigma_2)=\sigma_1|_{X'}\sigma_2|_{X'}\in E(\textup{FS}(X', Y'))$.

	Combining the above arguments,  $\textup{FS}(X', Y')  \cong \textup{FS}(X, Y)[ V_{\phi_i}]$ for any $i\in\{1,2,\cdots,(n-m)!\}$, and then the result  holds.
	\end{proof}
	%	\begin{lemma}\label{l61} 
		%		Let $s\geq2$ be a integer, $Y$ be a $s$-connected graph of order $n\geq3$. Then the connected components of $\textup{FS}(S_n, Y)$ are $s$-connected.
		%	\end{lemma}
	%	\begin{proof}
		%		Let $x_0$ be unique center vertex of $S_n$. For any $\sigma\in V(\textup{FS}(S_n, Y))$, there exists a corresponding $y_0\in V(Y)$ such that $\sigma(x_0)=y_0$. Since $Y$ is a $s$-connected graph, we have $d_Y(y_0)\geq s$. Without loss of generality, let $y_iy_0\in E(Y)$, where $1\leq i\leq s$. Then for $1\leq i\leq s$, we have $\sigma\sigma\circ (y_0\ y_i)\in E(\textup{FS}(S_n, Y))$ by $y_iy_0\in E(Y)$ and $x_0$ is a center vertex of $S_n$, which implies $d_{\textup{FS}(S_n, Y)}(\sigma)\geq s$. Thus $\delta(\textup{FS}(S_n, Y))\geq s$, which implies the connected components of $\textup{FS}(S_n, Y)$ is $s$-connected by Lemma \ref{lem611}.
		%	\end{proof}
		
	For any graph $X$ and $a,b\in V(X)$, we use $(a\ b)$ to denote the bijection $V(X)\rightarrow V(X)$ such that $(a\ b)(a)=b$, $(a\ b)(b)=a$ and $(a\ b)(c)=c$ for any $c\in V(X)\setminus \{a,b\}$.  
	
	For each vertex $\sigma$ of $\textup{FS}(S_n, C_n)$, we can read off the leaves of $S_n$ in the clockwise order that their images under $\sigma$ appear around $C_n$. In   \cite{dc2}, the authors  defined the content read off above as a cyclic ordering of the set of leaves of $S_n$, and obtained the following conclusion.
	
	\begin{lemma}{\rm(\!\!\cite{dc2})}\label{lem7}
		For any $\sigma,\tau\in V(\textup{FS}(S_n, C_n))$, $\sigma,\tau$ are in the same connected component of $\textup{FS}(S_n, C_n)$ if and only if they induce the same cyclic ordering of the leaves of $S_n$.
	\end{lemma}
	
	\begin{corollary}\label{l65}
		Let $\sigma\in V(\textup{FS}(S_n, C_n))$, $n\geq3$, and fix $x\in V(S_n)$, $y\in V(C_n)$. Then  after removing any vertex $\sigma_0 (\neq \sigma)$ of $\textup{FS}(S_n, C_n)$, there exists a vertex $\sigma'$ in the same connected component of $\textup{FS}(S_n, C_n)-\sigma_0$ as $\sigma$ such that $\sigma'(x)=y$.
	\end{corollary}
	
	\begin{proof}
		By Lemma \ref{lem611}, we have that the  connected components of $\textup{FS}(S_n,C_n)$ are $2$-connected. Moreover, by Lemma \ref{lem7}, each connected component of $\textup{FS}(S_n, C_n)$ has $n-1$ vertices $\sigma_1, \sigma_2,\cdots,\sigma_{n-1}$ such that $\sigma_1(x)=\sigma_2(x)=\cdots=\sigma_{n-1}(x)=y$. Thus each connected component of $\textup{FS}(S_n, C_n)-\sigma_0$ has at least $n-2 \ (\geq1)$ vertices $\sigma'_1, \sigma'_2,\cdots,\sigma'_{n-2}$ such that $\sigma'_1(x)=\sigma'_2(x)=\cdots=\sigma'_{n-2}(x)=y$. Therefore, the conclusion  holds.
		%it suffices to prove the result if each connected component of $\textup{FS}(S_n, C_n)$ has two vertices $\sigma_1, \sigma_2$ such that $\sigma_1(x)=\sigma_2(x)=y$. Without loss of generality, Let $V(S_n)=\{x_1,x_2,\cdots,x_n\}, V(C_n)=\{y_1,y_2,\cdots,y_n\}$ such that $\sigma (x_i)=y_i$, $x_1$ is centre vertex of $S_n$, and $y_1y_n,y_iy_{i+1}\in E(C_n)$ for $1\leq i\leq n-1$. If $x=x_1$
	\end{proof}
	\begin{lemma}\label{l62}
		Let $s\geq2$ be an integer, $Y$ be a $s$-connected graph of order $n(\geq3)$, $V$ be any subset of $ V(\textup{FS}(S_n, Y))$ with  $|V|=s-1$, $\sigma\in V(\textup{FS}(S_n, Y))\backslash V$,  and fix $x\in V(S_n)$ and $y\in V(Y)$. Then  after removing the $ s-1$ vertices in $V$, there exists a vertex $\sigma'$ in the same connected component of $\textup{FS}(S_n, Y)-V$ as $\sigma$ such that $\sigma'(x)=y$.
	\end{lemma} 
	
	\begin{proof}
		If $Y\cong C_n$, then the result holds by Corollary \ref{l65}. 
		
		If $Y\cong \theta_0$, then we can check by hand that the desired result holds. 
		
		 If $Y$ is non-bipartite and  $Y\not\in\{ C_n,\theta_0\}$, then $\textup{FS}(S_n, Y)$ is $s$-connected by $Y$ is  $s$-connected, Theorem \ref{th1} and Lemma \ref{lem611}, and thus $\textup{FS}(S_n, Y)-V$ is still a connected graph. On the other hand, there are $(n-1)!$ different vertex $\sigma'$ in $V(\textup{FS}(S_n, Y))$ such that $\sigma'(x)=y$, and the desired result holds by $(n-1)!\geq s$.
		
		If $Y$ is  bipartite and $Y\not\cong C_n$, now we show the desired result holds. Since $Y$ is a $s$-connected bipartite graph and $Y\not\cong C_4$, we have $n\geq 5$, $s\leq \delta(Y)\leq \frac{n}{2}$. Let $V(S_n)=V(Y)=\{1,2,\cdots,n\}$, $\{A_Y, B_Y\}$ be a bipartition of $Y$, $n$ be the center of $S_n$. Then $\{\{1,2,\cdots,n-1\}, \{n\}\}$ is a bipartition of $S_n$, denote by $A_X=\{1,2,\cdots,n-1\}$, $B_X=\{n\}$.
		
		 For each $\varphi\in V(\textup{FS}(S_n, Y))$,  $g(\varphi)=|\varphi(A_X)\cap A_Y|+\frac{sgn(\varphi)+1}{2}$ defined as  Lemma \ref{lem34}. 
			For any $i,j\in \{1,2,\cdots,n-1\} \setminus\{x\}$, let $\tau$ be a vertex of $\textup{FS}(S_n, Y)$ such that $\tau(x)=y$, and let $\tau_{ij}=\tau\circ (i\ j)$. Then $\tau_{ij}(x)=y$, $\tau(A_X)=\tau_{ij}(A_X)$ and $|\tau(A_X)\cap A_Y|=|\tau_{ij}(A_X)\cap A_Y|$. Therefore,  $g(\tau)$ and $g(\tau_{ij})$ have different parity since $\tau$ and $\tau_{ij}$ have different parity, and thus $\tau$ and $\tau_{ij}$ are in different connected components by Lemma \ref{lem34}. 
		
		On the other hand,    $\textup{FS}(S_n, Y)$ has exactly $2$ connected components $F_1$ and $F_2$ by Theorem \ref{th1}, both $F_1$ and $F_2$ are $s$-connected since $Y$ is $s$-connected and Lemma \ref{lem611}. Without loss of generality, we assume $\tau\in V(F_1)$, $\tau_{ij}\in V(F_2)$. Let $V_1= \{\tau_{ij} \ | \ i,j\in\{1,2,\cdots,n-1\}   \setminus\{x\} \}$. Then $V_1\subset V(F_2)$ and  $|V_1|\geq\frac{(n-2)(n-3)}{2}\geq n-2>\frac{n}{2}\geq s$. Similarly, we take $\tau_{ij}=\tau'$ for some $i,j\in \{1,2,\cdots,n-1\}   \setminus\{x\}$, and define $\tau'_{pq}=\tau'\circ (p\ q)$ for any $p,q\in \{1,2,\cdots,n-1\}   \setminus\{x\}$, $V_2= \{\tau'_{pq} \ | \ p,q\in\{1,2,\cdots,n-1\}   \setminus\{x\} \}$. Then we have $V_2\subset V(F_1)$ and  $|V_2|=|V_1|>s$. So for any $\sigma\in V(F_1)$ (or $\sigma\in V(F_2)$) and any $V$, there exists a vertex $\sigma'$ in the same connected component of $\textup{FS}(S_n, Y)-V$ as $\sigma$ such that $\sigma'(x)=y$ by $|V_1|=|V_2|>s$. 
		
			Combining the above arguments, we complete the proof.
	\end{proof}

		\begin{lemma}\label{zy}
		Let $s\geq2$ be an integer, $X$ and $Y$ be two connected graphs of order $n(\geq3)$ such that $\Delta(X)+\kappa(Y)\geq n+s-1$, $V$ be any subset of $ V(\textup{FS}(X, Y))$ with  $|V|=s-1$, $\sigma\in V(\textup{FS}(X, Y))\backslash V$,  and fix $x\in V(X)$ and $y\in V(Y)$. Then after removing $ s-1$ vertices in $V$, there exists a vertex $\sigma'$ in the same connected component of $\textup{FS}(X, Y)-V$ as $\sigma$ such that $\sigma'(x)=y$.
	\end{lemma}
	\begin{proof}
		Let $T$ be a spanning tree of $X$ with $\Delta(T)=\Delta(X)$. Then $\textup{FS}(T, Y)\preceq \textup{FS}(X, Y)$. Therefore, it suffices to prove the conclusion when $X$ is a tree.
		
		 Clearly, if  $\sigma(x)=y$, the result holds by taking $\sigma'=\sigma$. Now we prove the following Claim 1 holds. 
		 
		 \textbf{Claim 1. }Let $X$ be a tree and $\sigma(x)y\in E(Y)$. Then the result of Lemma \ref{zy} holds. 
		 \begin{proof}
		  If  $X\cong S_n$, then the conclusion holds by Lemma \ref{l62}.  
		 
		 If $X\not\cong S_n$, now we show the conclusion holds by the following arguments. Let $x_0\in V(X)$ such that $d_{X}(x_0)=\Delta (X)$. Then there exists a leaf $u$ of $X$ such that $u\not\in N_X[x_0]$ since  $X$ is a tree and $X\not\cong S_n$. We proceed by induction on $n$.
		 
		 If $n=3$, then $\Delta(X)+\kappa(Y)=4$ by $\Delta(X)+\kappa(Y)\geq 3+s-1$, $\Delta(X)\leq n-1$ and $   \kappa(Y)\leq n-1$,  thus $X\cong S_3$, a contradiction with $X\not\cong S_n$. If $n=4$, similar to $n=3$, we have $ X\cong P_4, Y\cong K_4$. Therefore, the result holds by $\textup{FS}(P_4, K_4)$ is $3$-connected.

		 Assume that Claim 1 is true for $n-1$ $(n\geq 5)$.  
		 We will proceed to prove the conclusion holds for $n$.
		 
		 Let $X'=X-u$, $Y'=Y-\sigma(u)$,  $V_1=\{\tau\in V \ | \ \tau(u)=\sigma(u)\}$, $V_2=\{\tau\in V \ | \ \tau(u)\neq \sigma(u)\}$, $V'_1=\{\tau|_{X'}  \ | \ \tau\in V_1\}$. Then $\sigma\notin V_1$, $V=V_1\cup V_2$, $\sigma|_{X'}\in V(\textup{FS}(X', Y'))$, $V'_1\subset V(\textup{FS}(X', Y')) $,  $|V'_1|=|V_1|\leq |V|= s-1$, $\Delta(X')=\Delta(X)$ by $d_{X'}(x_0)=d_X(x_0)$. Thus $\Delta(X')+\kappa(Y')\geq \Delta(X)+\kappa(Y)-1\geq(n-1)+s-1$. Furthermore, $X'$, $Y'$ are two connected graphs of order $n-1\ ( \geq4)$ by $X'=X-u$, $Y'=Y-\sigma(u)$ and $\kappa(Y')\geq s\geq 2$. So we can use the induction hypothesis for $X'$, $Y'$ and $\sigma|_{X'}$.
		 
		 \textbf{Case 1.}  $u\neq x$ and $\sigma(u)\neq y$.
		 
		 For $X'$, $Y'$,  $\sigma|_{X'}$, $V'_1$, $x\in V(X')$, $y\in V(Y')$, we have $\sigma|_{X'}\notin V'_1$ by $\sigma \notin V_1$, and $|V'_1|\leq s-1$, then by the induction hypothesis, there exists a vertex $\varphi$ in the same connected component $F'_1$ of $\textup{FS}(X', Y')-V'_1$ as $\sigma|_{X'}$ such that $\varphi(x)=y$.
		 
		 Let $V_3=\{\phi|^{X} \ | \ \phi\in V(F'_1)\}$,  $F_1=\textup{FS}(X, Y)[V_3]$, $\sigma'=\varphi|^{X}$. Then $\varphi|^{X}(u)=\sigma(u)$ by $\varphi\in V(\textup{FS}(X', Y'))\backslash V'_1$ and the definition of $\varphi|^{X}$, and $\sigma'$ is in the same connected component $F_2$ of $\textup{FS}(X, Y)-V_1$ as $\sigma$ such that $\sigma'(x)=\varphi|^{X}(x)=y$, where $V(F_1)\subset V(F_2)$.
		 
		 Furthermore, for any $\tau\in V_2$, it is clear $\tau\not\in V(F_1)$ and $\tau\neq \varphi|^{X}$ by $\tau(u)\neq \sigma(u)=\varphi|^{X}(u)$, then  $\sigma'=\varphi|^{X}$ is in the same connected component of $\textup{FS}(X, Y)-V$ as $\sigma$ such that $\sigma'(x)=y$ by $V=V_1\cup V_2$.
		 
		 \textbf{Case 2.} $u=x$.

		 It is obvious that there exists $u'\in V(X')$ such that $uu'\in E(X)$. Clearly, $u'$ is unique and certain, $\sigma(u)=\sigma(x)\neq y$ by $\sigma(x)y\in E(Y)$, and thus $y\in V(Y')$.
		 
		 For $X'$, $Y'$, $\sigma|_{X'}$, $V'_1$,  $u'\in V(X')$, $y\in V(Y')$,	by the induction hypothesis, there exists a vertex $\varphi_1$ in the same connected component $H'_1$ of $\textup{FS}(X', Y')-V'_1$ as $\sigma|_{X'}$ such that $\varphi_1(u')=y$.
		 
		 Let $V_4=\{\phi|^{X} \ | \ \phi\in V(H'_1)\}$, $H_1=\textup{FS}(X, Y)[V_4]$. Then $\varphi_1|^{X}$ (which implies $\varphi_1|^{X}(u)=\sigma(u)$) is in the same connected component $H_2$ of $\textup{FS}(X, Y)-V_1$ as $\sigma$ such that $\varphi_1|^{X}(u')=y$, where $V(H_1)\subset V(H_2)$. 
		 
		 Let $\sigma_1=\varphi_1|^{X}\circ (u \ u')$. Then $\sigma_1\varphi_1|^{X}\in E(\textup{FS}(X, Y))$ by $uu'\in E(X)$,  $\sigma_1(u)\sigma_1(u')=\varphi_1|^{X}(u')\varphi_1|^{X}(u) =y\sigma(u)\in E(Y)$ and Definition \ref{d1}. Furthermore, we have $\sigma_1(u)=y\neq \sigma(u)$, which implies $\sigma_1\not\in V_1$. 
		 
		 If $\sigma_1\not\in V_2$,  then $\sigma_1$ is in the same connected component  of $\textup{FS}(X, Y)-V$ as $\varphi_1|^{X}$ (and thus as $\sigma$) such that $\sigma_1(u)=\sigma_1(x)=y$ by $V_2\cap V(H_1)=\emptyset$, $V=V_1\cup V_2$, and $\sigma_1\varphi_1|^{X}\in E(\textup{FS}(X, Y))$.  We take $\sigma'=\sigma_1$, the result holds.
		 
		 If $\sigma_1\in V_2$, then $V_2\neq \emptyset$, and $|V'_1|=|V_1|\leq s-2$ by $|V|=s-1$ and $V=V_1\cup V_2$. Let $V_{1,1}=V_1\cup \{\varphi_1|^{X}\}$, $V'_{1,1}=V'_1\cup \{\varphi_1\}$. Then $|V_{1,1}|=|V'_{1,1}|\leq s-1$.  For $X'$, $Y'$, $\sigma|_{X'}$, $V'_{1,1}$,  $u'\in V(X')$, $y\in V(Y')$, by the induction hypothesis, there exists a vertex $\varphi_2$ in the same connected component $H'_3$ of $\textup{FS}(X', Y')-V'_{1,1}$ as $\sigma|_{X'}$ such that $\varphi_2(u')=y$.
		 
		 Let $V_5=\{\phi|^{X} \ | \ \phi\in V(H'_3)\}$, $H_3=\textup{FS}(X, Y)[V_5]$. Then $\varphi_2|^{X}$ (which implies $\varphi_2|^{X}(u)=\sigma(u)$) is in the same connected component $H_4$ of $\textup{FS}(X, Y)-V_{1,1}$ as $\sigma$ such that $\varphi_2|^{X}(u')=y$, where  $V(H_3)\subset V(H_4)$.
		 
		 Let $\sigma_2=\varphi_2|^{X}\circ (u \ u')$.  Then $\sigma_2\varphi_2|^{X}\in E(\textup{FS}(X, Y))$ by $uu'\in E(X)$ and $\sigma_2(u)\sigma_2(u')=\varphi_2|^{X}(u')\varphi_2|^{X}(u)=y\sigma(u)\in E(Y)$ and Definition \ref{d1}. Furthermore, we have $\sigma_2(u)=y\neq \sigma(u)$, which implies $\sigma_2\not\in V_{1,1}$.
		 
		 If $\sigma_2\not\in V_2$,   then $\sigma_2$ is in the same connected component  of $\textup{FS}(X, Y)-V$ as $\sigma$ such that $\sigma_2(u)=\sigma_2(x)=y$ by  $V_2\cap V(H_3)=\emptyset$, $V=V_1\cup V_2\subset V_{1,1}\cup V_2$, and  $\sigma_2\varphi_2|^{X}\in E(\textup{FS}(X, Y))$.  We take $\sigma'=\sigma_2$, the result holds.
		 
		 If $\sigma_2\in V_2$,  we take $V_{1,2}=V_1\cup \{\varphi_1|^{X},\varphi_2|^{X}\}$, $V'_{1,2}=V'_1\cup \{\varphi_1, \varphi_2\}$, obtain $\varphi_3, H'_5, V_6, H_6, \varphi_3|^{X}$ and $\sigma_3$ by the similar discussion as the case of $\sigma_1\in V_2$.
		 
		 In general, by repeating the above discussion,  there exists some $i$ such that $2\leq i\leq s-1-|V_1|$ and $\sigma_i\notin V_{1,i-1}$.
		 
		 If $\sigma_i\not\in V_2$, then  $\sigma_i$ is in the same connected component  of $\textup{FS}(X, Y)-V$ as $\sigma$ such that $\sigma_i(u)=\sigma_i(x)=y$ by $V_2\cap V(H_{2i-1})=\emptyset$, $V=V_1\cup V_2\subset V_{1,i-1}\cup V_2$, and  $\sigma_i\varphi_i|^{X}\in E(\textup{FS}(X, Y))$.  We take $\sigma'=\sigma_i$, the result holds. 
		 
		 If  $\sigma_{i}\in V_2$, then $\sigma_1,\cdots,\sigma_{i-1}\in V_2$. Let $V_{1,i}=V_1\cup \{\varphi_1|^{X},\cdots,\varphi_i|^{X} \}$, $V'_{1,i}=V'_1\cup \{\varphi_1,\cdots,\varphi_i\}$.  Then $|V_{1,i}|=|V'_{1,i}|\leq  s-1$. For $X'$, $Y'$, $\sigma|_{X'}$, $V'_{1,i}$,  $u'\in V(X')$, $y\in V(Y')$, by the induction hypothesis, there exists a vertex $\varphi_{i+1}$ in the same connected component $H'_{2i+1}$ of $\textup{FS}(X', Y')-V'_{1,i}$ as $\sigma|_{X'}$ such that $\varphi_{i+1}(u')=y$.
		 
		 Let $V_{i+4}=\{\phi|^{X} \ | \ \phi\in V(H'_{2i+1})\}$, $H_{2i+1}=\textup{FS}(X, Y)[V_{i+4}]$. 	Then $\varphi_{i+1}|^{X}$ is in the same connected component $H_{2i+2}$ of $\textup{FS}(X, Y)-V_{1,i}$ as $\sigma$ such that $\varphi_{i+1}|^{X}(u')=y$, where  $V(H_{2i+1})\subset V(H_{2i+2})$. 
		 
		 Let $\sigma_{i+1}=\varphi_{i+1}|^{X}\circ (u \ u')$.  Then $\sigma_{i+1}\varphi_{i+1}|^{X}\in E(\textup{FS}(X, Y))$ by $uu'\in E(X)$,  $\sigma_{i+1}(u)\sigma_{i+1}(u')=\varphi_{i+1}|^{X}(u')\varphi_{i+1}|^{X}(u)=y\sigma(u)\in E(Y)$, and Definition \ref{d1}. Furthermore, we have $\sigma_{i+1}(u)=y\neq \sigma(u)$, which implies $\sigma_{i+1}\not\in V_{1,i}$.
		 
		 If $\sigma_{i+1}\not\in V_2$, then $\sigma_{i+1}$ is in the same connected component  of $\textup{FS}(X, Y)-V$ as $\sigma$ such that $\sigma_{i+1}(u)=y$ by  $V_2\cap V(H_{2i+1})=\emptyset$, $V=V_1\cup V_2\subset V_{1,i}\cup V_2$ and  $\sigma_{i+1}\varphi_{i+1}|^{X}\in E(\textup{FS}(X, Y))$.  We take $\sigma'=\sigma_2$, the result holds.
		 
		 If $\sigma_{i+1}\in V_2$, we repeat the above arguments, and there must exist some $j$ such that $\sigma_{j}\not\in V_2$. In fact, 
		 $\sigma_{s-|V_1|}\not\in V_2$. Otherwise,  we have $\sigma_1,\cdots,\sigma_{s-|V_1|}\in V_2$, a contradiction with $|V_2|=|V|-|V_1|=s-1-|V_1|$. Therefore, there exists a vertex $\sigma_j$ in the same connected component of $\textup{FS}(X, Y)-V$ as $\sigma$ such that $\sigma_j(x)=y$, where $j\in\{1,2,\cdots,s-|V_1|\}$. We take $\sigma'=\sigma_j$,  the result holds.
		 
		 \textbf{Case 3.} $u\neq x$ and $\sigma(u)= y$.
		 
		 It is obvious that there exists a vertex $y'\in V(Y')$ such that $yy'\in E(Y)$. Since
		 $\sigma(u)y'\in E(Y)$, we can repeat the arguments of Case $2$ with  $y$ replaced by $y'$. Thus there exists a vertex $\overline{\sigma}$ in the same connected component of $\textup{FS}(X, Y)-V$ as $\sigma$ such that $\overline{\sigma}(u)=y'$.
		 
		 Now, for $X,Y,x,y,\overline{\sigma},V$, 
		 %Clearly, we have $u$ is a leaf of $X$, $u\not\in N_X[x_0]$, $u\neq x$, and $\overline{\sigma}(u)\neq y$. 
		 we can repeat the arguments of Case 1 with $\sigma$ replaced by $\overline{\sigma}$, and there exists a vertex $\sigma'$ in the same connected component of $\textup{FS}(X, Y)-V$ as $\overline{\sigma}$ (and thus as $\sigma$) such that $\sigma'(x)=y$. 
		 
		 Combining the above arguments, we complete the proof of Claim 1.
		 \end{proof}

		Now we show the result of Lemma \ref{zy} holds.
		
		 If $\sigma(x)y\in E(Y)$, then the result holds by Claim 1.
		
		If  $\sigma(x)y\notin E(Y)$, then there exists a path $P$ from $\sigma(x)$ to $y$ since $Y$ is connected. Let $P=\sigma(x)y_1y_2\cdots y_ty$. Then there exists a vertex $\sigma'_1$ in the same connected component of $\textup{FS}(X, Y)-V$ as $\sigma$ such that $\sigma'_1(x)=y_1$ by Claim 1 and $\sigma(x)y_1\in E(\textup{FS}(X, Y))$. Furthermore, for each  $2\leq i\leq t$, there exists a  vertex $\sigma'_i$ in the same connected component of $\textup{FS}(X, Y)-V$ as $\sigma'_{i-1}$ such that $\sigma'_i(x)=y_i$ by Claim 1 and $y_{i-1}y_i\in E(\textup{FS}(X, Y))$. Finally, there exists a vertex $\sigma'$ in the same connected component of $\textup{FS}(X, Y)-V$ as $\sigma'_t$ such that $\sigma'(x)=y$ by Claim 1 and $y_ty\in E(\textup{FS}(X, Y))$. Therefore,  $\sigma'$ is in the same connected component of $\textup{FS}(X, Y)-V$ as $\sigma$. We complete the proof.
	\end{proof}
	
\vskip 0.4cm
		By definition of $s$-connected, we have the following observation immediately. %so we omit the proof.
	\begin{Observation}\label{p2}
		Let $G$ be a graph, $V(G)=U\cup W$ be a partition.  Then we have
		
		{\rm(i)} if  both $G[U]$  and  $G[W]$ are $s$-connected, and there exist $s$ vertex-disjoint edges between $G[U]$ and $G[W]$, then $G$ is $s$-connected;
		
		{\rm(ii)} if $G[U]$ is $s$-connected,   for any subset  $V$ of $V(G)$ with  $|V|=s-1<|W|$ and any vertex $v \in W\backslash V$,  there exists a vertex of $U$ which is in the same connected component  of $G-V$ as  $v$, then $G$ is $s$-connected;
		
		{\rm(iii)} if $G$ has  exactly two connected components $G_1$ and $G_2$, $U=U_1\cup U_2$ is a partition, $G[U]$ has exactly two connected components $G[U_1]$ and $G[U_2]$, where both $G[U_1]$ and $G[U_2]$ are $s$-connected,  for any subset  $V$ of $V(G)$ with  $|V|=s-1< |W|$ and any vertex $v \in W\backslash V$,  there exists a vertex of $U$ which is in the same connected component  of $G-V$ as  $v$,  then  both $G_1$ and $G_2$ are $s$-connected.
	\end{Observation}
	
		We denote $W_n$ as the wheel graph with $n$ vertices.

	\begin{lemma}\label{l2}
		Let $Z$ be a connected graph of order $n (\geq4)$, $\Delta(Z)= n-2$. Then  $\textup{FS}(Z,W_n)$ is $2$-connected. 
	\end{lemma}
	\begin{proof}
   Clearly, we have $\textup{Dand}_{2,n-2}\preceq Z$. By Lemma \ref{lem3}, it  suffices to show that $\textup{FS}(\textup{Dand}_{2,n-2},W_n)$ is $2$-connected.
   
   	Without loss of generality, we take $x_0,x_1\in V(\textup{Dand}_{2,n-2})$ with $d_{\textup{Dand}_{2,n-2}}(x_1)=\Delta(\textup{Dand}_{2,n-2})$ and $ x_0x_1\notin E(\textup{Dand}_{2,n-2})$, $y_0\in V(W_n)$ with $ d_{W_n}(y_0)=3$. If $n=4$, then $\textup{FS}(\textup{Dand}_{2,2} -x_0,W_4-y_0)$ is $2$-connected by direct calculation. If $n\geq5$, then $\textup{FS}(\textup{Dand}_{2,n-2}-x_0,W_n-y_0)$ is $2$-connected by $\textup{Dand}_{2,n-2}-x_0\cong S_{n-1}$, $\kappa(W_n-y_0)=\delta(W_n-y_0)=2$, Theorem \ref{th1} and Lemma \ref{lem611}. 
   	
   	Let $\overline{V}=\{\sigma\in V(\textup{FS}(\textup{Dand}_{2,n-2},W_n))|\sigma(x_0)=y_0\}$. Then  $\textup{FS}(\textup{Dand}_{2,n-2},W_n)[\overline{V}]\cong \textup{FS}(\textup{Dand}_{2,n-2}-x_0,W_n-y_0)$ by Lemma \ref{l63}, and thus  $\textup{FS}(\textup{Dand}_{2,n-2},W_n)[\overline{V}]$ is  $2$-connected.
   	
   	Let $X=\textup{Dand}_{2,n-2}$, $Y=W_n$, $V$ be any subset of $\textup{FS}(\textup{Dand}_{2,n-2},W_n)$ with $|V|=s-1=1$. Then $\Delta(X)+\kappa(Y)= n+1$. We fix $x_0,y_0$, which are defined as above, for any $\sigma \in V(\textup{FS}(X,Y))\backslash V$, by applying Lemma \ref{zy}, we know there exists a vertex $\sigma'$ in the same connected component of $\textup{FS}(X,Y)-V$ as $\sigma$ such that $\sigma'(x_0)=y_0$. On the other hand, we take $G=\textup{FS}(X,Y)$, $U=\overline{V}$ and $V(\textup{FS}(X,Y))=U\cup W$. Then $G[U]=\textup{FS}(X,Y)[\overline{V}]$ is $2$-connected, and thus $G=\textup{FS}(X,Y)$ is $2$-connected by the above arguments and  {\rm(ii)} of Observation \ref{p2}. Therefore, we complete the proof by $X=\textup{Dand}_{2,n-2}$, $Y=W_n$.

	\end{proof}
	
	\begin{lemma}\label{b1}
		Let $\theta_1$ be shown in Figure \ref{Figure 1}. Then $\textup{FS}(\textup{Dand}_{2,6},\theta_1)$ is  $2$-connected.
	\end{lemma}
	\begin{proof}
		Clearly, we have $\Delta(\textup{Dand}_{2,6})=6$ and $ \kappa(\theta_1)=3$. Without loss of generality, we take $x_0,x_1\in V(\textup{Dand}_{2,6})$ with $d(x_1)=\Delta(\textup{Dand}_{2,6})$ and $ x_0x_1\notin E(\textup{Dand}_{2,6})$, $y_0,y_1\in V(\theta_1)$ with $ d(y_1)=5$ and $ y_0y_1\notin E(\theta_1)$. Then  $\theta_1-y_0$ is 2-connected, non-bipartite, $\theta_1-y_0\not\in\{C_7,\theta_0\}$ and $\textup{Dand}_{2,6}-x_0\cong S_{7}$. Thus $\textup{FS}(\textup{Dand}_{2,6}-x_0,\theta_1-y_0)$ is  $2$-connected by Theorem \ref{th1} and Lemma \ref{lem611}. 
		
		Similar to the proof of Lemma \ref{l2}, we have that  $\textup{FS}(\textup{Dand}_{2,6},\theta_1)$ is  $2$-connected.
	\end{proof}
	
	\begin{lemma}\label{bz1}
		Let $G$ be a non-bipartite graph of order $n$. Then $G\cong C_n$ and $n$ is  odd  if $G-v$ is a bipartite graph for each $v\in V(G)$.
	\end{lemma}
	\begin{proof}
		Let $G_1,\cdots, G_t$ be all connected components of $G$. Then there exists an odd cycle $C_r$  in some connected component $G_i$  for $1\leq i\leq t$. Without loss of generality, we suppose $i=1$.
		
		 Firstly, we show $G_1$ is the unique connected component of $G$. Otherwise, $t\geq2$. Let $v_0\in V(G_t)$. Then $G-v_0$ is  bipartite. But $G-v_0$ contains the odd cycle $C_r$, and thus $G-v_0$ is non-bipartite, a contradiction. Therefore, we have $G= G_1$. 
		 
		 Secondly, we show $r=n$. Otherwise, there exists  $v_1\in V(G)\backslash V(C_r)$, and $G-v_1$ is  non-bipartite since $G-v_1$ contains the odd cycle $C_r$, a contradiction.  Therefore, we have $C_n \preceq G$. Without loss of generality, we suppose $C_n=u_1u_2\cdots u_nu_1$.
		 
		 Finally, we show $G\cong C_n$. Otherwise,  there  exist $i,j\in \{1,2,\cdots,n\}$ such that $u_iu_j\in E(G)$ and $2\leq|j-i|\leq n-2$. Let $i< j$ and $j-i=k$. Then $C'=u_iu_{i+1}\cdots u_ju_i$ with length $k+1$ and $C''=u_ju_{j+1}\cdots u_nu_1\cdots u_i u_j$ with length $n-k+1$ are two cycles of $G$ with different parity.	Thus $G-v_2$ still contains an odd cycle and $G-v_2$ is non-bipartite, where $$v_2=\begin{cases}
		 	u_{j+1},&\text{if} \ k \ \text{is even and} \ j<n, \\
		 	u_1,&\text{if} \ k \ \text{is even and} \ j= n, \\
		 	u_{i+1}, &\text{if} \ k \ \text{is odd},
		\end{cases}$$ a contradiction. Therefore, $G\cong C_n$ and $n$ is odd.
	\end{proof}

	\begin{lemma}\label{l3}
		Let  $Y$ be a graph of order $n (\geq4)$, $\kappa(Y)=3$. Then
		
		{\rm(i)} if  $Y$ is a bipartite graph, then $\textup{FS}(\textup{Dand}_{2,n-2},Y)$ has exactly two connected components, where each connected component is  $2$-connected;
		
		{\rm(ii)} if  $Y$ is non-bipartite, then $\textup{FS}(\textup{Dand}_{2,n-2},Y)$ is $2$-connected.
	\end{lemma}
	\begin{proof}
			Let $x_0,x_1\in V(\textup{Dand}_{2,n-2})$ with $d(x_1)=\Delta(\textup{Dand}_{2,n-2})  =n-2$ and $ x_0x_1\notin E(\textup{Dand}_{2,n-2})$. Then $\textup{Dand}_{2,n-2}-x_0\cong S_{n-1}$.
			
		{\rm(i)} Clearly, there exists $y_0\in V(Y)$ such that $\kappa(Y-y_0)=2$ and $Y-y_0$ is  bipartite. 
		
		Firstly, we show   $Y-y_0\not\in\{C_{n-1},\theta_0\}$.  It is clear that $Y-y_0\not\cong \theta_0$ by $Y$ is  bipartite  and $\theta_0$ is non-bipartite. If $Y-y_0\cong C_{n-1}$, then $Y\cong W_n$ by $\kappa(Y)=3$, a contradiction with $Y$ is  bipartite. 
		
		Secondly, we show  $\textup{FS}(\textup{Dand}_{2,n-2},Y)$ has exactly two connected components and each connected component is $2$-connected. 
		
Clearly, $\textup{FS}(\textup{Dand}_{2,n-2}-x_0,Y-y_0)$ has exactly two connected components $F_1$ and $F_2$  by $\textup{Dand}_{2,n-2}-x_0\cong S_{n-1}$ and Theorem \ref{th1}, where both $F_1$ and $F_2$ are $2$-connected by $\kappa(Y-y_0)=2$ and Lemma \ref{lem611}.  
		Let $\overline{V}=\{\sigma\in V(\textup{FS}(\textup{Dand}_{2,n-2},Y)|\sigma(x_0)=y_0\}$. Then  $\textup{FS}(\textup{Dand}_{2,n-2},Y)[\overline{V}]\cong\textup{FS}(\textup{Dand}_{2,n-2} -x_0,Y-y_0)= F_1\cup F_2$ by Lemma \ref{l63}, and 
		thus $\textup{FS}(\textup{Dand}_{2,n-2},Y)$ has at most two connected components by $\Delta(\textup{Dand}_{2,n-2})+\kappa(Y)=n+1$ and Lemma \ref{zy}. Furthermore,  $\textup{FS}(\textup{Dand}_{2,n-2},Y)$ has exactly two connected components $F_3$ and $F_4$  by Lemma \ref{lem32}.
		
		Now we show both $F_3$ and $F_4$  are $2$-connected. Let $X=\textup{Dand}_{2,n-2}$,  $V$ be any subset of $V(\textup{FS}(\textup{Dand}_{2,n-2},Y))$ with $|V|=1$.  We fix $x_0,y_0$, which are defined as above, for any $\sigma \in V(\textup{FS}(X,Y))\backslash V$, by applying Lemma \ref{zy}, we know there exists a vertex $\sigma'$ in the same connected component of $\textup{FS}(X,Y)-V$ as $\sigma$ such that $\sigma'(x_0)=y_0$. On the other hand, we take $G=\textup{FS}(X,Y)$, $G_1=F_3$, $G_2=F_4$, $U=\overline{V}$ and $V(\textup{FS}(X,Y))=U\cup W$. Then $G[U]=\textup{FS}(X,Y)[\overline{V}]\cong F_1\cup F_2$, and thus both $F_3$ and $F_4$ are $2$-connected by the above arguments and  {\rm(iii)} of Observation \ref{p2}.

		  Therefore,  we complete the proof by  $\textup{FS}(\textup{Dand}_{2,n-2},Y)=F_3\cup F_4$, where both $F_3$ and $F_4$ are $2$-connected.

		{\rm(ii)} Clearly, there exists $y_1\in V(Y)$ such that $Y-y_1$ is non-bipartite. Otherwise, $Y\cong C_n$ with odd $n$  by Lemma \ref{bz1}, and thus $\kappa(Y)=2$, a contradiction with $\kappa(Y)=3$.
		
		If $Y-y_1\cong C_{n-1}$, then $Y\cong W_n$ by $\kappa(Y)=3$. Thus $\textup{FS}(\textup{Dand}_{2,n-2},Y)$ is $2$-connected by $\Delta(\textup{Dand}_{2,n-2})=n-2$ and Lemma \ref{l2}.
		
		If $Y-y_1\cong \theta_0$, then $\theta_1\preceq Y$. Thus $\textup{FS}(\textup{Dand}_{2,6},Y)$ is $2$-connected by Lemma \ref{b1} and Lemma \ref{lem3}.
		
		Now we consider the case where $Y-y_1\not\in\{C_{n-1},\theta_0\}$. It is obvious that  $\textup{FS}(\textup{Dand}_{2,n-2}-x_0,Y-y_1)$ is $2$-connected by $\kappa(Y-y_1)\geq2$, $\textup{Dand}_{2,n-2}-x_0\cong S_{n-1}$, Theorem \ref{th1} and Lemma \ref{lem611}. Similar to the proof of Lemma \ref{l2}, we have that  $\textup{FS}(\textup{Dand}_{2,n-2},Y)$ is  $2$-connected.
	\end{proof}

\begin{comment}
\begin{figure}[!h]
	\centering
	\begin{tikzpicture}
		\node[anchor=south west,inner sep=0] (image) at (0,0) {\includegraphics[width=0.5\textwidth]{g2}};
		\begin{scope}[
			x={(image.south east)},
			y={(image.north west)}
			]
			\node [black, font=\bfseries] at (0.2,0.2) {$\textup{Dand}_{2,n-2}^+$ };
			\node [black, font=\bfseries] at (0.7,0.2) {$\textup{Dand}_{2,n-2}^{\dagger}$ };
		\end{scope}
	\end{tikzpicture}
	\vspace{-1.8em}
	\caption{Graphs $\textup{Dand}_{2,n-2}^+$ and $\textup{Dand}_{2,n-2}^{\dagger}$.}\label{Figure 2}
\end{figure}
\end{comment}

	%\hspace{-0.4em}	\textbf{Theorem \ref{mt1}} 	Let $X$ and $Y$ be two connected graph of order $n$, $Y$ is non-bipartite, $Y\not\cong C_n,\theta_0$, $\Delta(X)+\kappa(Y)\geq n+1$. Then  $\textup{FS}(X,Y)$ is connected. In particular, if $\kappa(Y)\geq \lceil \frac{n+1}{2}\rceil$, $n\geq 5$, $\Delta(X)+\kappa(Y)\geq n+1$, then $\textup{FS}(X,Y)$ is connected.
		
		%{\rm(i)} If $\Delta(X)\geq \lceil \frac{n+1}{2}\rceil$, $Y$ is non-bipartite, $Y\not\cong C_n,\theta_0$, then $\textup{FS}(X,Y)$ is connected.
		
		%{\rm(ii)} If $\Delta(X)< \lceil \frac{n+1}{2}\rceil$, then $\textup{FS}(X,Y)$ is connected.
		\vskip 0.4cm
			Let $X$ be a graph, $x_1,x_2\in V(X)$, we use $X+x_1x_2 $ to denote a new graph from $X$ by adding   the edge $x_1x_2$ to  $X$ for  $x_1x_2\notin E(X)$, and  $ X-x_1x_2$ to denote  removing the edge $x_1x_2$ from $X$ for $ x_1x_2\in E(X)$.

			\begin{lemma}\label{b2}
				Let $X$ and $Y$ be two connected bipartite graphs of order $n (\geq3)$, $x_1$ and $x_2$  belong to the same part of $V(X)$. If $\textup{FS}(X,Y)$ has exactly two connected components $F_1, F_2$,  and  each connected component is $l$-connected, where $1\leq l\leq (n-1)(n-2)$, then $\textup{FS}(X+x_1x_2,Y)$ is $l$-connected.
			\end{lemma}
			
			\begin{proof}
				 If $3\leq n\leq 5$, then the result can be obtained by hand calculation. In the following,  we consider the case when $n\geq 6$.
				
				Without loss of generality, let   $V(X)=V(Y)$, $\{A_{X},B_{X}\}$ be a  bipartition of $X$, $x_1,x_2\in A_{X}$, $\{A_Y,B_Y\}$ be a  bipartition of $Y$, $g(\sigma)=|\sigma(A_X)\cap A_Y|+\frac{sgn(\sigma)+1}{2}$ as  Lemma \ref{lem34},  $V(\textup{FS}(X,Y))=V_1\cup V_2$, where $V_1=\{\sigma|g(\sigma)\equiv 0\ (\text{mod} \ 2)\}, V_2=\{\sigma|g(\sigma)\equiv 1\ (\text{mod} \ 2)\}$. Then 
				%any $\sigma_1, \sigma_2\in V_1$ in a connected component $F_1$ of $\textup{FS}(X,Y)$ and any $\sigma_3, \sigma_4\in V_2$ in another connected component $F_2$ of $\textup{FS}(X,Y)$ 
				$V_1=V(F_1)$, $V_2=V(F_2)$ by Lemma \ref{lem34} and $\textup{FS}(X,Y)=F_1\cup F_2$.
				
				Let $G=\textup{FS}(X+x_1x_2,Y)$. Then $V(G)=V_1\cup V_2$ and $F_1\preceq G[V_1]$, $F_2\preceq G[V_2]$. Therefore, $G[V_1]$ and $G[V_2]$ are $l$-connected since $F_1$ and $F_2$ are $l$-connected. Now we show there exist
				 $(n-1)(n-2)$ vertex-disjoint edges between $G[V_1]$ and $G[V_2]$, which implies $G=\textup{FS}(X+x_1x_2,Y)$ is $s$-connected by {\rm(i)} of Observation \ref{p2}. 
				
				Since $Y$ is a connected bipartite graph, we have $|E(Y)|\geq n-1$.  Without loss of generality, let $e_i$ be any edge of $ E(Y)$, where $1\leq i\leq |E(Y)|$. Then there exists $\sigma_i\in V(\textup{FS}(X+x_1x_2,Y))$ such that $\sigma_i(x_1)\sigma_i(x_2)=e_i$, $\sigma'_i=\sigma_i\circ(x_1 \ x_2)$. Then $\sigma_i\sigma'_i\in E(\textup{FS}(X+x_1x_2,Y))$. Moreover,   $\sigma_i$ and $\sigma'_i$ have different parity, and $|\sigma_i(A_X)\cap A_Y|=|\sigma'_i(A_X)\cap A_Y|$, which implies $g(\sigma_i)$ and $ g(\sigma'_i)$  have different parity. Thus one of $\sigma_i$ and $\sigma'_i$ is in $V_1$,  the other is in $V_2$ by Lemma \ref{lem34}, and $\sigma_i\sigma'_i$ is an edge between $G[V_1]$ and $G[V_2]$.
				
				 Clearly,  $|A_{X}|\geq2$ by $x_1,x_2\in A_{X}$. Then we complete the proof by the following four cases.
				
					\textbf{Case 1.} $|A_X|=2$.
					
				 Then $|B_X|= n-2$. Without loss of generality, let $B_X=\{x_j|3\leq j\leq n\}$,  $\sigma_{ij_1j_2}=\sigma_i \circ(x_{j_1} \ x_{j_2})$, $\sigma'_{ij_1j_2}=\sigma'_i \circ(x_{j_1} \ x_{j_2})$, $3\leq j_1<j_2\leq n$. Then $\sigma_{ij_1j_2}(x_1)\sigma_{ij_1j_2}(x_2)=e_i$ and $\sigma'_{ij_1j_2}=\sigma_{ij_1j_2}\circ(x_1 \ x_2)$. Similar to the proof of above, we have that $\sigma_{ij_1j_2}\sigma'_{ij_1j_2}\in E(\textup{FS}(X+x_1x_2,Y))$, and one of $\sigma_{ij_1j_2}$ and $\sigma'_{ij_1j_2}$ is in $V_1$,  the other is in $V_2$. Therefore, for each $i\in \{1,\cdots, |E(Y)| \}$,  there are at least $1+ \binom{n-2}{2}$ vertex-disjoint edges between $G[V_1]$ and $G[V_2]$ by $|B_X|= n-2$, and thus there are at least $|E(Y)|(1+\binom{n-2}{2})>(n-1)( n-2)$ vertex-disjoint edges between $G[V_1]$ and $G[V_2]$ by $|E(Y)|\geq n-1$ and $1\leq i\leq |E(Y)|$.
				 
				 	\textbf{Case 2.} $|A_X|=3$.
				
				Similar to the proof of Case 1, we have at least $|E(Y)|(1+\binom{n-3}{2})\geq (n-1)( n-2)$ vertex-disjoint edges between $G[V_1]$ and $G[V_2]$.%Without loss of generality, let $B_X=\{x_i|4\leq i\leq n\}$,  $\sigma_{ij}=\sigma_1 \circ(x_i \ x_j)$, $\sigma'_{ij}=\sigma'_1 \circ(x_i \ x_j)$, $4\leq i,j\leq n$, $i\neq j$. Similarly, $\sigma_{ij}\sigma'_{ij}\in E(\textup{FS}(X+x_1x_2,Y))$ and one of $\sigma_{ij}$ and $\sigma'_{ij}$ is in $F_1$, and the other is in $F_2$. Thus we have $1+\frac{(n-3)(n-4)}{2} (\geq n-2)$ vertex-disjoint edges in $F_1$ and $F_2$.
				
				\textbf{Case 3.} $|A_X|\geq4$ and $|B_X|\geq 2$.

					 Let $|A_X\setminus \{x_1,x_2\}|=r$, then $|B_X|=n-2-r$.  Similarly, we can adopt $x_{j_1}, x_{j_2}\in A_X\setminus \{x_1,x_2\}$, or $x_{j_1}, x_{j_2}\in B_X$, then there are  at least $|E(Y)|(1+\binom{r}{2}+\binom{n-2-r}{2})$ vertex-disjoint edges between $G[V_1]$ and $G[V_2]$. It is obvious that $|E(Y)|(1+\binom{r}{2}+\binom{n-2-r}{2})\geq (n-1)(1+\binom{r}{2}+\binom{n-2-r}{2})\geq (n-1)(n-2)$.
				
				\textbf{Case 4.} $|A_X|=n-1$ and $|B_X|= 1$.
				
				 Clearly,  $|A_X\setminus \{x_1,x_2\}|=n-3$, then there are at least $|E(Y)|(1+\binom{n-3}{2})\geq(n-1)( n-2)$ vertex-disjoint edges between $G[V_1]$ and $G[V_2]$ by similar arguments as above.

				We complete the proof.
			\end{proof}
			
			\begin{corollary}\label{c3}
				Let $Y$ be a $2$-connected graph of order $n\geq3$, $Y\notin \{C_n, \theta_0\}$, $\delta(Y)\geq s$. Then $\textup{FS}(S^+_n,Y)$  is $s$-connected.
			\end{corollary}
			\begin{proof}
				 If $Y$ is non-bipartite, then $\textup{FS}(S_n,Y)$  is $s$-connected by Theorem \ref{th1}  and  Lemma \ref{lem611}, and thus $\textup{FS}(S^+_n,Y)$ is $s$-connected by Lemma \ref{lem3}.
				 
				   If $Y$ is  bipartite, then $\textup{FS}(S_n,Y)$ has exactly two connected components and each connected component is  $s$-connected by Theorem \ref{th1} and Lemma \ref{lem611}, and thus  $\textup{FS}(S^+_n,Y)$ is $s$-connected by Lemma \ref{b2}.
			\end{proof}

Now we give the proof of Theorem \ref{mt21} and Theorem \ref{mt1}.

		\noindent\textbf{\textit{Proof of Theorem \ref{mt21}.}}
		We show {\rm(i)} and {\rm(ii)} in sequence.
		
		{\rm(i)} 	Let $T$ be a spanning tree of $X$ with $\Delta(T)=\Delta(X)$. Then $\textup{FS}(T, Y)\preceq \textup{FS}(X, Y)$. Therefore, it suffices to prove the conclusion when $X$ is a tree.
		
		If $X\cong S_n$, then  $\kappa(Y)\geq s$ by  $\Delta(X)+\kappa(Y)\geq n+s-1$ and $\Delta(X)=n-1$,  and thus $\textup{FS}(X,Y)$  has exactly  two connected components and each connected component is  $s$-connected by $Y\not\cong C_n$, Theorem \ref{th1} and Lemma \ref{lem611}. %If $n=\Delta(X)+2$, then $\textup{Dand}_{2,n-2}\preceq X$ and $\kappa(Y)=3$. Thus $\textup{FS}(X,Y)$ has at most two connected components and each connected component is  $2$-connected by Lemma \ref{l3}. 
		
		If $X\not\cong S_n$, then $\Delta(X)\leq n-2$ and $\kappa(Y)\geq n+s-1-\Delta(X)\geq s+1\geq3$. Now we show the conclusion holds by the following arguments. Let $x_1\in V(X)$ such that $d_{X}(x_1)=\Delta (X)$. Then there exists a leaf $x_0$ of $X$ such that $x_0\not\in N_X[x_1]$ since  $X$ is a tree and $X\not\cong S_n$,  thus $\Delta(X-x_0)=\Delta(X)$ and $X-x_0$ is connected. Clearly, there exists $y_0\in V(Y)$ such that $Y-y_0\not\cong C_{n-1}$. Otherwise,  $Y-y\cong C_{n-1}$ for any $y\in V(Y)$, then $Y\cong W_4$ by $\kappa(Y)\geq3$, a contradiction.  We proceed by induction on $n$.
		
		If $3\leq n\leq 5$, then there  does not exist a tree $X\not\cong S_n$ and a bipartite graph $Y$ such that $\Delta(X)+\kappa(Y)\geq n+s-1$.
		
		 If $n=6$,  then $\Delta(X)+\kappa(Y)=7$ by $\Delta(X)+\kappa(Y)\geq 6+s-1$, $\Delta(X)\leq n-2$ and $
		\kappa(Y)\leq \lfloor \frac{n}{2}\rfloor$,  thus $s=2$,  $\Delta(X)= n-2=4$ and $
		\kappa(Y)= \lfloor \frac{n}{2}\rfloor=3$. Therefore,  $X\cong \textup{Dand}_{2,4}$ and $ Y\cong K_{3,3}$, the result holds since $\textup{FS}(\textup{Dand}_{2,4}, K_{3,3})$ has exactly  two connected components and each connected component is  $2$-connected by  direct calculation. 
		
		Assume that the result is true for $n-1$ $(n\geq 7)$.  
		We will proceed to prove the conclusion holds for $n$.
		
		Clearly, we have $\Delta(X-x_0)+\kappa(Y-y_0)=\Delta(X)+\kappa(Y-y_0)\geq n+s-2$. By the induction hypothesis, $\textup{FS}(X-x_0,Y-y_0)$ has exactly two connected components and each connected component is  $s$-connected. Let $\overline{V}=\{\sigma\in V(\textup{FS}(X,Y))|\sigma(x_0)=y_0\}$. Then  $\textup{FS}(X,Y)[\overline{V}]\cong \textup{FS}(X-x_0,Y-y_0)$ by Lemma \ref{l63}, thus $\textup{FS}(X,Y)[\overline{V}]$ has exactly two connected components and each connected component is  $s$-connected. Similar to proof of {\rm(i)} of Lemma \ref{l3}, we have that  $\textup{FS}(X,Y)$ has exactly two connected components and each connected component is  $s$-connected.
		
			{\rm(ii)}  Let $X\cong \textup{Dand}_{n-\Delta,\Delta}$, $\kappa(Y)=\kappa$, $\Delta+\kappa\leq n$, $1\leq\kappa\leq \lfloor \frac{n}{2}\rfloor$. Clearly, $\textup{FS}(\textup{Dand}_{n-\Delta,\Delta},Y)$ has at least two connected components by Lemma \ref{lem32}.    If $\textup{FS}(\textup{Dand}_{n-\Delta,\Delta},Y)$ has exactly two connected components, then $\textup{FS}(\textup{Lollipop}_{n-\Delta,\Delta},Y)$ is connected by Lemma \ref{b2} and Lemma \ref{lem3}, a contradiction with Theorem \ref{th4}. Thus $\textup{FS}(\textup{Dand}_{n-\Delta,\Delta},Y)$ has at least three connected components.
		 {\hfill $\blacksquare$ \par}
		 
		\vspace{0.5em}

\noindent\textbf{\textit{Proof of Theorem \ref{mt1}.}}
	 We complete the proof by the following two cases.
	%	If $n=3$, then $Y\cong K_3$, and thus $\textup{FS}(X,Y)$ is connected by Lemma \ref{g}. If $n=4$, then $\Delta(X)=3$ or $Y\cong K_4$. Clearly, $\textup{FS}(X,Y)$ is connected if $Y\cong K_4$. In addition, $S_4\preceq X$, $\kappa(Y)\geq2$, $Y$ is non-bipartite, and $Y\not\cong C_n,\theta_0$ if $\Delta(X)=3$. Then $\textup{FS}(X,Y)$ is connected by Theorem \ref{th1}. The following, we will consider the case when $n\geq 5$.

	\textbf{Case 1.} $Y$ is non-bipartite.
	
	Similarly,  it suffices to prove the result when $X$ is a tree.
	
	If  $X\cong S_n$, then $Y$ is $s$-connected, and thus the conclusion holds by Theorem \ref{th1}, Lemma \ref{lem611} and $\delta(Y)\geq \kappa(Y)\geq s$.

	If $X\not\cong S_n$,  then $\Delta(X)\leq n-2$, and $\kappa(Y)\geq n+s-1-\Delta(X)\geq s+1\geq3$.
	
		If $\kappa(Y)=3$, then $s=2$, $\Delta(X)=n-2$ by $\Delta(X)+\kappa(Y)\geq n+s-1$ and $\Delta(X)\leq n-2$, and thus $X\cong \textup{Dand}_{2,n-2}$. Therefore, the desired result holds by (\rm ii) of Lemma \ref{l3}.
	
	If $\kappa(Y)\geq 4$, then $n\geq 5$ by $\delta(Y)\geq \kappa(Y)\geq4$. Now we  proceed the proof by induction on $n$.    If $n=5$, then  $X\in\{P_5, \textup{Dand}_{2,3}\} $ and $Y\cong K_5$, and thus the result holds since both $\textup{FS}(P_5,K_5)$ and $\textup{FS}(\textup{Dand}_{2,3},K_5)$ are $4$-connected by direct calculation.
	
	Assume that conclusion is true for $n-1$ $(n\geq 6)$.  We now show it also holds for $n$.
	
	 Let $x_1\in V(X)$ such that $d_{X}(x_1)=\Delta (X)$. Then there exists a leaf $x_0$ of $X$ such that $x_0\not\in N_X[x_1]$ since  $X$ is a tree and $X\not\cong S_n$,  thus $\Delta(X-x_0)=\Delta(X)$ and $X-x_0$ is connected.
	 Clearly, for any $y\in V(Y)$, we have  $\kappa(Y-y)\geq \kappa(Y)-1\geq3$, and there exists $y_0\in V(Y)$ such that  $Y-y_0$ is a non-bipartite graph by Lemma \ref{bz1} and $Y\not\cong C_n$,  and  $\Delta(X-x_0)+\kappa(Y-y_0)\geq n-1+s-1$. Furthermore, we have $Y-y_0\notin \{C_{n-1}, \theta_0\}$ by $\kappa(Y)\geq4$, and then for $s\geq 2, X-x_0, Y-y_0$, we have 
 $\textup{FS}(X-x_0,Y-y_0)$  is $s$-connected 	
 by the induction hypothesis. 
 
 Let  $V$ be any subset of $\textup{FS}(X,Y)$ with $|V|=s-1$, $\overline{V}=\{\tau\in V(\textup{FS}(X,Y))| \tau(x_0)=y_0\}$. Then $\textup{FS}(X,Y)[\overline{V}]\cong \textup{FS}(X-x_0,Y-y_0)$ by Lemma \ref{l63}, and thus $\textup{FS}(X,Y)[\overline{V}]$  is $s$-connected. 	  We fix $x_0,y_0$,  which are defined as above, for any $\sigma \in V(\textup{FS}(X,Y))\backslash V$, by applying Lemma \ref{zy}, we know there exists a vertex $\sigma'$ in the same connected component of $\textup{FS}(X,Y)-V$ as $\sigma$ such that $\sigma'(x_0)=y_0$. On the other hand, we take $G=\textup{FS}(X,Y)$, $U=\overline{V}$ and $V(\textup{FS}(X,Y))=U\cup W$. Then $G[U]=\textup{FS}(X,Y)[\overline{V}]$ is $s$-connected, and thus $G=\textup{FS}(X,Y)$ is $s$-connected by the above arguments and  {\rm(ii)} of Observation \ref{p2}.

	\textbf{Case 2.}  $Y$ is bipartite and $X$ is non-bipartite.

		Let $T$ be a spanning tree of $X$ with $\Delta(T)=\Delta(X)$. Then $\Delta(T)+\kappa(Y)\geq n+s-1$, and thus  $\textup{FS}(T,Y)$  has exactly two connected components and each connected component is  $s$-connected by Theorem \ref{mt21}. 
	
	Let $\{A_{T},B_{T}\}$ be a  bipartition of $T$. Then $V(X)=V(T)=A_{T} \cup B_{T}$. Furthermore, there exist $\{x_2,x_3\}\subseteq A_T$ or $B_T$ such that $x_2x_3\in E(X)$ and $x_2x_3\not\in E(T)$. Otherwise, $X$ is a bipartite graph, a contradiction. Clearly,  $T+x_2x_3\preceq X$. Thus  $\textup{FS}(T+x_2x_3,Y)$ is $s$-connected by the above arguments, $s\leq \Delta(T)+\kappa(Y)-n+1\leq n-1\leq (n-1)(n-2)$ and Lemma \ref{b2}. Therefore, we have $\textup{FS}(X,Y)$ is $s$-connected by $T+x_2x_3\preceq X$ and  Lemma \ref{lem3}.
			 {\hfill $\blacksquare$ \par}

	 By Theorem \ref{mt1} and $\kappa(C_n)=\kappa(\theta_0)=2$, we have the following  corollary immediately.
	 \begin{corollary}\label{mt11}
	 	Let $s\geq2$ be an integer, $X$ and $Y$ be two connected graphs of order $n$, among which at least one is non-bipartite,  $\Delta(X)+\kappa(Y)\geq n+s-1$. Then 
	 	
	 	{\rm(i)} if $s=2$, $\Delta(X)=n-1$ and  $Y\not\in\{ C_n,\theta_0\}$,  then  $\textup{FS}(X,Y)$ is $2$-connected;
	 	
	 	{\rm(ii)} if $s=2$ and $\Delta(X)\leq n-2$, then $\textup{FS}(X,Y)$ is $2$-connected;
	 	
	 	{\rm(iii)} if $s\geq3$, then  $\textup{FS}(X,Y)$ is $s$-connected.
	 \end{corollary}
	 
	 	 \begin{proposition}\label{p4}
	 	For any integers $\Delta$ and $\kappa$ that satisfy $\Delta+\kappa\leq n$, $2\leq\Delta\leq n-1$, and $1\leq\kappa\leq n-2$, there exists a connected  graph $X$ of order $n$ such that $\Delta(X)=\Delta$,  for any connected  graph $Y$ with $\kappa(Y)=\kappa$, the graph $\textup{FS}(X,Y)$ is disconnected.
	 \end{proposition}
	 \begin{proof}
	 	Let $X\in DL_{n-\Delta,\Delta}$ with $2\leq \Delta=\Delta(X)\leq n-1$. Then $1\leq\kappa(Y)=\kappa\leq n-2$ and $\Delta(X)+\kappa(Y)\leq n$. By Lemma \ref{th4} and Lemma \ref{lem3}, we conclude that $\textup{FS}(X,Y)$ is disconnected. 
	 \end{proof}

Clearly, Proposition \ref{p4} implies that bound on $\Delta(X)+\kappa(Y)$  of Theorem 1.7 and Corollary \ref{mt11}, $n+1$ is the best. 

	\section{\textbf{Some applications of Theorem \ref{mt1}}}
		\subsection{\textbf{The proof of Theorems \ref{mt2} and \ref{mt3}}}
		
	%	\begin{lemma}\label{l41}
	%		Let $X (\not\cong K_n)$ be a graph of order $n$, $s$ be a positive integer. If $\delta(X)\geq  \lceil \frac{n+s-2}{2}\rceil$, then $X$ is $s$-connected.
	%	\end{lemma}
		
	%	\begin{proof}
	%		By $X \not\cong K_n$, we have $\lceil \frac{n+s-2}{2}\rceil\leq n-2$, which implies $s\leq n-2$. For any $x_1x_2\notin E(G)$, there exists a positive integer $r$ such that we can obtain $X+x_1x_2$ by deleting $r$ edges from $K_n$. Without loss of generality, let $K_n-x_{k_1}x_{l_1}=X_1$, $X_i-x_{k_i}x_{l_i}=X_i$, $X_r-x_1x_2=X$, where $2\leq i\leq r$. Clearly, we have $\delta(X_j)\geq\delta(X)\geq \lceil \frac{n+s-2}{2}\rceil$, which implies $d_X(x_1)+d_X(x_2)\geq n+s-2$, $d_{X_j}(x_{k_j}) +d_{X_j}(x_{l_j})\geq n+s-2$ for $1\leq j\leq r$. By applying Lemma \ref{lem5} $r+1$ times, we can successively obtain $ X_1, X_2, \cdots, X_r$ and $X$ are $s$-connected. We complete the proof.
	%	\end{proof}
	
	%By Lemma \ref{l41}, we have the following conclusion immediately.
	
%	\begin{corollary}\label{co1}
%		Let $X$ be a graph of order $n$. Then $\delta(X)\leq \lceil \frac{n+\kappa(X)-2}{2}\rceil$.
%	\end{corollary}
	
%	\begin{proof}
%		If $X\cong K_n$, then $\delta(X)=n-1= \lceil \frac{2n-3}{2}\rceil=\lceil \frac{n+\kappa(X)-2}{2}\rceil$. In the following, we consider the case when $X\not\cong K_n$.
		
%		Suppose to the contrary, if $\delta(X)\geq \lceil \frac{n+\kappa(X)+2-2}{2}\rceil$, then $X$ is  $(\kappa(X)+2)$-connected by Lemma \ref{l41}, a contradiction. Therefore, we have $\delta(X)\leq \lceil \frac{n+\kappa(X)-2}{2}\rceil$.
%	\end{proof}

	\hspace{2em}By Theorem \ref{mt2}, we  have that Theorem \ref{mt3} holds immediately. So we only show Theorem \ref{mt2}.

	\vspace{0.5em}
	
\noindent\textbf{\textit{Proof of Theorem \ref{mt2}.}}  
{\rm(i)} Clearly, $Y$ is  non-bipartite,  $Y\not\in\{ C_n,\theta_0\}$ and $\Delta(X)+\kappa(Y)=n+\Delta(X)-1$ by $Y\cong K_n$, and thus $\textup{FS}(X,Y)$ is $\Delta(X)$-connected by  Theorem \ref{mt1}.

 {\rm(ii)}  By Lemma \ref{lem5} and $Y\not\cong K_n$, we have $\kappa(Y)\geq 2\delta(Y)+2-n$. Then
 $\Delta(X)+\kappa(Y)\geq \Delta(X)+2\delta(Y)+2-n\geq n+s$  by $\Delta(X)+2\delta(Y)\geq2n+s-2$. Moreover, we have $\delta(Y)\geq \frac{2n+s-2-\Delta(X)}{2}\geq \frac{n+1}{2}$. Thus $Y$ is non-bipartite and  $Y\not\in\{ C_n,\theta_0\}$. Therefore, $\textup{FS}(X,Y)$ is $(s+1)$-connected by  Theorem \ref{mt1}.
 {\hfill $\blacksquare$ \par}

	In fact, {\rm(i)} of Theorem \ref{mt2} improves upon the result of Lemma \ref{g}.
		
		\subsection{\textbf{The answer of Problem \ref{p1} and further results}}

		\begin{theorem}\label{mt4}
		Let $n, k\geq 2$ be two integers,  $n\geq k+1$, $X\in\textup{DL}_{n-k,k}$,   $Y$ be a graph of order $n$. Then  
		
		{\rm(i)} if $n=k+1$ and $X\cong\textup{Dand}_{1,n-1}$, then $\textup{FS}(X,Y)$ is connected if and only if $Y$ is $2$-connected non-bipartite graph and $Y\not\in\{C_n, \theta_0\}$;
		
       {\rm(ii)} if $n=k+1$, $X\not\cong\textup{Dand}_{1,n-1}$ and $Y\not\in\{C_n, \theta_0\}$, then $\textup{FS}(X,Y)$ is connected if and only if $Y$ is $2$-connected; 
		
		{\rm(iii)} if $k+2\leq n\leq 2k-2$ and $X\cong\textup{Dand}_{n-k,k}$, then $\textup{FS}(X,Y)$ is connected if and only if $Y$ is  non-bipartite and $(n-k+1)$-connected;

		{\rm(iv)} if $k+2\leq n\leq 2k-2$ and $X\not\cong\textup{Dand}_{n-k,k}$, then $\textup{FS}(X,Y)$ is connected if and only if $Y$ is $(n-k+1)$-connected;
		
		{\rm(v)} if $n\geq 2k-1$ and $n\neq k+1$, then $\textup{FS}(X,Y)$ is connected if and only if $Y$ is $(n-k+1)$-connected;
		
		{\rm(vi)}  $\textup{FS}(X,Y)$ is connected if and only if $\textup{FS}(X,Y)$ is $2$-connected.
	\end{theorem}

\begin{proof}
	{\rm(i)}   If $X\cong\textup{Dand}_{1,n-1}$, then the result holds by Theorem \ref{th1} and $\textup{Dand}_{1,n-1}=S_n$.
	
	Now we prove  {\rm(ii)}-{\rm(v)}. 
	Clearly, $\Delta(X)=k$ by $X\in\textup{DL}_{n-k,k}$.
	
	Necessity.
	
	Let $\textup{FS}(X,Y)$ be a connected graph. If $\kappa(Y)\leq n-k$ and $n\geq k+1$, then $\textup{FS}(X,Y)$ is disconnected by Lemma \ref{th4} and Lemma \ref{lem3}, a contradiction. If $Y$ is a bipartite graph for  $k+2\leq n\leq 2k-2$ and $X\cong\textup{Dand}_{n-k,k}$, then $\textup{FS}(X,Y)$ is disconnected by Lemma \ref{lem32}, a contradiction. Now we show the sufficiency of {\rm(ii)}-{\rm(v)}.
	
	Sufficiency.

   {\rm(ii)} 	Let $Y$ be $2$-connected. By $n=k+1$ and  $X\not\cong\textup{Dand}_{1,n-1}$, we have $\Delta(X)=n-1$ and $X$ is non-bipartite, and thus  $\textup{FS}(X,Y)$ is connected by $\Delta(X)+\kappa (Y)\geq n+1$, $Y\not\in\{C_n, \theta_0\}$ and Theorem  \ref{mt1}. 
	
	{\rm(iii)}  Let $Y$ be non-bipartite, $(n-k+1)$-connected. Then $\Delta(X)+\kappa (Y)\geq n+1$. Thus $\textup{FS}(X,Y)$ is connected by  Theorem \ref{mt1}. 
	
{\rm(iv)} Let $Y$ be $(n-k+1)$-connected. Clearly,  $X$ is non-bipartite by  $X\not\cong\textup{Dand}_{n-k,k}$, and thus  $\textup{FS}(X,Y)$ is connected by $\Delta(X)+\kappa (Y)\geq n+1$ and Theorem  \ref{mt1}.

{\rm(v)} If $n\geq 2k-1$ and $Y$ is $(n-k+1)$-connected, then $Y$ is non-bipartite by $\delta(Y)\geq\kappa(Y)\geq n-k+1\geq \frac{n+1}{2}$, and thus $\textup{FS}(X,Y)$ is connected by $\Delta(X)+\kappa (Y)\geq n+1$ and Theorem  \ref{mt1}.

Finally, we show {\rm(vi)}.

{\rm(vi)} It is obvious that $\textup{FS}(X,Y)$ is connected if $\textup{FS}(X,Y)$ is $2$-connected. If $\textup{FS}(X,Y)$ is connected, then $Y$ is $(n-k+1)$-connected and at least one of $X$ and $Y$ is non-bipartite graph by {\rm(i)}-{\rm(v)}. Clearly, we have $\Delta(X)+\kappa (Y)\geq n+1$. Thus  $\textup{FS}(X,Y)$ is $2$-connected by Theorem \ref{mt1}.  We complete the proof.
\end{proof}  

\begin{remark}\label{r1}
	When $n\geq 8$, $X\in \textup{DL}_{n-k,k}$ with $n\geq k+1$, the connectedness of $\textup{FS}(X,Y)$  can be characterized completely by Theorem \ref{mt4} and Lemma \ref{lem33}.
\end{remark}

Based on Theorem \ref{mt4} and the conclusions of $n\geq 2k-1$ and $n=k$ in \cite{wang3}, we provide a complete answer to Problem \ref{p1} as follows.

\begin{proposition}\label{p5}
	Let $n\geq k\geq 2$. Then we have 
	
	{\rm(i)} for $n\geq 2k-1$, $\textup{FS}(\textup{Lollipop}_{n-k,k},Y)$ is connected if and only if $\textup{FS}(\textup{Dand}_{n-k,k},Y)$  is connected.
	
	{\rm(ii)} for $k\leq n\leq 2k-2$,  the statement of Problem \ref{p1} does not hold.
\end{proposition}

		  Next, we  provide a sufficient  condition in terms of $\kappa(X) + \kappa(Y)$ such that $\textup{FS}(X,Y)$ is connected.
		  
		  \begin{theorem}\label{th6}
		  	Let $X$ and $Y$ be two  connected graphs of order $n$. 
		  	
		  	{\rm(i)} If $\kappa(X)+\kappa(Y)\geq n+1$, then $\textup{FS}(X,Y)$ is $2$-connected. 
		  	
		  	{\rm(ii)} If $\kappa(X)+\kappa(Y)= n$ and $n$ is an odd number, then $\textup{FS}(X,Y)$ is connected. 
		  	
		  	{\rm(iii)} If $n(\geq5)$ is an even  number, then there exist connected graphs $X$ and $Y$ of order $n$ such that $\kappa(X)+\kappa(Y)= n$, and $\textup{FS}(X,Y)$ is disconnected.
		  	
		  	{\rm(iv)} If $n\geq5$, then there exist connected graphs $X$ and $Y$ of order $n$ such that $\kappa(X)+\kappa(Y)= n-1$, and $\textup{FS}(X,Y)$ is disconnected.
		  \end{theorem}
		  
		  \begin{proof}
		  	{\rm(i)} It is clear that $\Delta(X)+\kappa(Y)\geq n+1$ by  $\kappa(X)+\kappa(Y)\geq n+1$ and $\Delta(X)\geq \delta(X)\geq \kappa(X)$. Now we show $\textup{FS}(X,Y)$ is $2$-connected.
		  	
		  	 If $Y\in\{C_n, \theta_0\}$, then $\kappa(Y)=2$, and thus $\kappa(X)\geq n-1$, which implies $X\cong K_n$. Therefore, we have $\textup{FS}(X,Y)$ is $2$-connected by {\rm(i)} of Theorem \ref{mt2}. 
		  	 
		  	 If $Y\not\in\{C_n, \theta_0\}$, without loss of generality, we assume $\kappa(X)\geq \lceil \frac{n+1}{2}\rceil$. Then $X$ is non-bipartite by $\kappa(X)\leq \delta(X)$. Thus $\textup{FS}(X,Y)$ is $2$-connected by  Theorem \ref{mt1}. 
		  	 
		  {\rm(ii)} We have $\kappa(X)$ or $\kappa(Y)$ is an odd number by	$\kappa(X)+\kappa(Y)= n$ and $n$ is an odd number. Without loss of generality, we suppose  $\kappa(X)$  is an odd number. Then $\Delta(X)\geq \kappa(X)+1$. Otherwise, $\Delta(X)=\kappa(X)\leq \delta(X)$, which implies $d_X(x)=\kappa(X)$ for any $x\in V(X)$, and  $\sum\limits_{x\in V(X)}d_X(x)=n\kappa(X)$ is an odd number, a contradiction. Thus $\Delta(X)+\kappa(Y)\geq \kappa(X)+\kappa(Y)+1\geq n+1$. Furthermore, let $\kappa(X)\geq \kappa(Y)$. Then $\kappa(X)>\frac{n}{2}$ since $n$ is odd, and thus $X$ is non-bipartite.  We complete the proof by the following two cases.
		  	
		  	\textbf{Case 1}: $Y\not\in\{C_n, \theta_0\}$.
		  	
		  	It is obvious that $\textup{FS}(X,Y)$ is $2$-connected by Theorem \ref{mt1}.
		  	
		  	\textbf{Case 2}: $Y\in\{C_n, \theta_0\}$.
		  	
		  	Clearly, we have  $\kappa(Y)=2$. Thus $\kappa(X)=n-2$, which implies $K_n-\frac{n-1}{2}e\preceq X$. 
		  	
		  	If $Y\cong C_n$, then $\textup{FS}(K_n-\frac{n-1}{2}e,Y)$ is connected by Lemma \ref{lem33}. Thus $\textup{FS}(X,Y)$ is connected by Lemma \ref{lem3}.
		  	
		  	If $Y\cong \theta_0$, then there exists $y_0\in V(Y)$ such that $Y-y_0\cong C_6$. Meanwhile, there exists $z_0\in V(K_7-3e)$ such that $K_7-3e-z_0\cong K_6-2e$, where $d_{K_7-3e}(z_0)=5$. Clearly, we have $\textup{FS}(K_7-3e-z_0, Y-y_0)$ is connected by Lemma \ref{lem33}. Similar to the proof of Lemma \ref{l2}, we have that  $\textup{FS}(K_7-3e, Y)$ is connected, which implies $\textup{FS}(X, Y)$ is connected.

		  	{\rm(iii)} Let $X\cong C_n$, $Y\cong K_n-\frac{n}{2}e$.  Clearly, we have $\kappa(K_n-\frac{n}{2}e)=n-2$  by $n (\geq 5)$ is an even number, which implies $\kappa(X)+\kappa(Y)=\kappa(C_n)+\kappa(K_n-\frac{n}{2}e)= n$.  By  $\overline{Y}=\frac{n}{2}K_2$ and Lemma \ref{lem33}, we have that $\textup{FS}(X,Y)$ is disconnected.
		  	
		  	{\rm(iv)} 	 Let $X\cong C_n$, $\overline{Y}\cong C_n$. Then $Y$ is   connected  by $n\geq 5$ with $\kappa(Y)=n-3$, and thus $\kappa(X)+\kappa(Y)=n-1$. By Lemma \ref{lem33}, we have that $\textup{FS}(X,Y)$ is disconnected.
		  \end{proof}

Clearly, {\rm(iii)} and {\rm(iv)} of Theorem \ref{th6} imply that conditions  of {\rm(i)} and {\rm(ii)} of Theorem \ref{th6} are the best. 

Furthermore, by {\rm(vi)} of Theorem \ref{mt4}, when $X\in \textup{DL}_{n-k,k}$ with $n\geq k+1$, $\textup{FS}(X,Y)$ is connected if and only if $\textup{FS}(X,Y)$ is $2$-connected. A natural question is what graphs other than in $\textup{DL}_{n-k,k}$ satisfy this property?

	\section*{\bf Funding}
	
	\hspace{1.5em} This work is  supported by the National Natural Science Foundation of China (Grant Nos. 12371347, 12271337).

\end{document}